\documentclass[english]{article}
\usepackage[T1]{fontenc}
\usepackage[latin9]{inputenc}
\usepackage{float}
\usepackage{amsmath}
\usepackage{amsthm}
\usepackage{amssymb}
\usepackage{graphicx}
\usepackage{esint}

\makeatletter
  \theoremstyle{definition}
  \newtheorem{defn}{\protect\definitionname}
 \theoremstyle{definition}
  \newtheorem{example}{\protect\examplename}
\theoremstyle{plain}
\newtheorem{thm}{\protect\theoremname}
  \theoremstyle{remark}
  \newtheorem{rem}{\protect\remarkname}
  \theoremstyle{plain}
  \newtheorem{prop}{\protect\propositionname}

\makeatother

\usepackage{babel}
  \providecommand{\definitionname}{Definition}
  \providecommand{\examplename}{Example}
  \providecommand{\propositionname}{Proposition}
  \providecommand{\remarkname}{Remark}
\providecommand{\theoremname}{Theorem}

\begin{document}

\title{Dynamic Rotation and Stretch Tensors \\
from a Dynamic Polar Decomposition\thanks{To appear in the \emph{Journal of the Mechanics and Physics of Solids}} }

\author{George Haller\thanks{Email address: georgehaller@ethz.ch}}

\maketitle
\begin{center}
Institute of Mechanical Systems\\
Department of Mechanical and Process Engineering, \\
ETH Zürich, Leonhardstrasse 21, 8092 Zürich, Switzerland
\par\end{center}
\begin{abstract}
The local rigid-body component of continuum deformation is typically
characterized by the rotation tensor, obtained from the polar decomposition
of the deformation gradient. Beyond its well-known merits, the polar
rotation tensor also has a lesser known dynamical inconsistency: it
does not satisfy the fundamental superposition principle of rigid-body
rotations over adjacent time intervals. As a consequence, the polar
rotation diverts from the observed mean material rotation of fibers
in fluids, and introduces a purely kinematic memory effect into computed
material rotation. Here we derive a generalized polar decomposition
for linear processes that yields a unique, dynamically consistent
rotation component, the dynamic rotation tensor, for the deformation
gradient. The left dynamic stretch tensor is objective, and shares
the principal strain values and axes with its classic polar counterpart.
Unlike its classic polar counterpart, however, the dynamic stretch
tensor evolves in time without spin. The dynamic rotation tensor further
decomposes into a spatially constant mean rotation tensor and a dynamically
consistent relative rotation tensor that is objective for planar deformations.
We also obtain simple expressions for dynamic analogues of Cauchy's
mean rotation angle that characterize a deforming body objectively.
\end{abstract}

\section{Introduction}

In continuum mechanics, the now classic procedure for isolating the
rotational component of the deformation gradient is the polar decomposition.
To describe this decomposition, we consider a deformation field $\mathbf{X}_{t_{0}}^{t}\colon\mathbf{\mathbf{x}}(t_{0})\mapsto\mathbf{\mathbf{x}}(t)$
defined on a spatial domain $\mathcal{B}(t_{0})\subset\mathbb{R}^{3}$
over the time interval $\left[t_{0},t_{1}\right]\subset\mathbb{R}.$
The trajectories $\mathbf{\mathbf{x}}(t)$ depend on the initial time
$t_{0}$ and the initial position $\mathbf{x}_{0}$, but this will
be suppressed for notational simplicity. By the polar decomposition
theorem, the deformation gradient $\mathbf{F}_{\tau}^{t}=\mathbf{\nabla}\mathbf{X}_{\tau}^{t}$
(with $\tau,t\in\left[t_{0},t_{1}\right]$) has unique left and right
decompositions of the form
\begin{equation}
\mathbf{\mathbf{F}}_{\tau}^{t}=\mathbf{R}_{\tau}^{t}\mathbf{U}_{\tau}^{t}=\mathbf{V}_{\tau}^{t}\mathbf{R}_{\tau}^{t},\label{eq:PD}
\end{equation}
with a proper orthogonal matrix $\mathbf{\mathbf{R}}_{\tau}^{t}$
and symmetric, positive definite matrices $\mathbf{\mathbf{U}}_{\tau}^{t}$
and $\mathbf{\mathbf{V}}_{\tau}^{t}$ (Truesdell \& Noll \cite{Truesdell65}).
Although customarily suppressed in their notation, the rotation and
stretch tensors do depend on the time interval $[\tau,t]$. We keep
this dependence in our notation for later purposes. We also emphasize
that we consider general non-autonomous deformation fields for which
the velocity field $\mathbf{\dot{X}}_{t_{0}}^{t}$ may depend explicitly
on the current time $t$, which therefore cannot be set to zero at
arbitrary intermediate configurations for convenience.

In finite-strain theory, the polar rotation tensor $\mathbf{\mathbf{R}}_{\tau}^{t}$
is interpreted as solid-body rotation, while $\mathbf{\mathbf{U}}_{\tau}^{t}$
and $\mathbf{\mathbf{V}}_{\tau}^{t}$ are referred to as right and
left stretch tensors between the times $\tau$ and $t$ (Truesdell
and Noll \cite{Truesdell65}). The tensor $\mathbf{\mathbf{R}}_{\tau}^{t}$
is generally obtained from \eqref{eq:PD} after $\mathbf{U}_{\tau}^{t}$
is computed as the principal square root of the right Cauchy--Green
strain tensor $\mathbf{\mathbf{C}}_{\tau}^{t}=\left(\mathbf{\mathbf{F}}_{\tau}^{t}\right)^{T}\mathbf{\mathbf{F}}_{\tau}^{t}$. 

As Boulanger and Hayes \cite{boulanger} (see also Jaric et al. \cite{jaric06})
point out, there are in fact infinitely many possible rotation-stretch
decompositions of the form $\mathbf{\mathbf{F}}_{\tau}^{t}=\mathbf{\Omega}\mathbf{\Delta}$,
where $\mathbf{\Omega}\in SO(3)$ is a rotation and $\Delta$ is a
non-degenerate tensor whose singular values and singular vectors coincide
with the eigenvalues and eigenvectors of $\mathbf{\mathbf{C}}_{\tau}^{t}$.
Indeed, an infinity of such decompositions can be generated from any
given one by selecting an arbitrary rotation $\mathbf{\mathbf{\Xi}}\in SO(3)$
and letting
\begin{equation}
\mathbf{\mathbf{F}}_{\tau}^{t}=\mathbf{\mathbf{\hat{\Omega}}}\mathbf{\hat{\Delta}},\qquad\mathbf{\mathbf{\mathbf{\hat{\Omega}}}}=\mathbf{\Omega}\mathbf{\Xi},\qquad\mathbf{\hat{\Delta}}=\mathbf{\Xi}^{T}\mathbf{\Delta}.\label{eq:infinitely_many}
\end{equation}
Out of these infinitely many rotation-stretch decompositions, the
left polar decomposition in \eqref{eq:PD} is uniquely obtained by
requiring $\mathbf{\Delta}$ to be symmetric and positive definite.
This convenient choice has important advantages, but is by no means
necessary for capturing the strain invariants of the deformation,
given that $\mathbf{\mathbf{C}}_{s}^{t}=\mathbf{\Delta}^{T}\mathbf{\Delta}$
is always the case, even for a non-symmetric choice of $\mathbf{\Delta}$.
In addition, there is no a priori physical reason why the stretching
component of the deformation gradient should be symmetric. In particular,
requiring $\mathbf{\Delta}=\mathbf{\Delta}^{T}=\mathbf{U}_{s}^{t}$
does not render $\mathbf{\dot{\Delta}\mathbf{\Delta^{-1}}}$ symmetric.
In other words, the evolution of $\mathbf{U}_{\tau}^{t}$ is not spin-free.

The main advantage of the polar decomposition \eqref{eq:PD} is an
appealing geometric interpretation of the particular rotation generated
by $\mathbf{\mathbf{R}}_{\tau}^{t}$. Indeed, $\mathbf{\mathbf{R}}_{\tau}^{t}$
rotates the eigenvectors of $\mathbf{\mathbf{C}}_{\tau}^{t}$ into
eigenvectors of the left Cauchy--Green strain tensor $\mathbf{\mathbf{B}}_{\tau}^{t}=\mathbf{\mathbf{F}}_{\tau}^{t}\left(\mathbf{\mathbf{F}}_{\tau}^{t}\right)^{T}$,
or equivalently, into eigenvectors of $\mathbf{\mathbf{C}}_{t}^{\tau}$
(Truesdell and Noll \cite{Truesdell65}). This property distinguishes
$\mathbf{\mathbf{R}}_{\tau}^{t}$ as a highly plausible geometric
rotation\emph{ }component for the deformation gradient between the
times $\tau$ and $t$. A further remarkable feature of the polar
rotation tensor is that $\mathbf{\mathbf{R}}_{\tau}^{t}$ represents,
among all rotations, the closest fit to $\mathbf{\mathbf{F}}_{\tau}^{t}$
in the Frobenius matrix norm (Grioli \cite{grioli40}, Neff et al.
\cite{neff14}).

These geometric advantages of $\mathbf{\mathbf{R}}_{\tau}^{t}$, relative
to a fixed initial time $\tau$ and a fixed end time $t$, however,
also come with a disadvantage for times evolving within $[\tau,t]$:
polar rotations computed over adjacent time intervals are not additive.
More precisely, for any two sub-intervals $[\tau,s]$ and $[s,t]$
within $\left[\tau,t\right]$, we have 
\begin{equation}
\mathbf{R}_{\tau}^{t}\neq\mathbf{R}_{s}^{t}\mathbf{R}_{\tau}^{s},\label{eq:no_product}
\end{equation}
unless $\mathbf{\mathbf{U}}_{s}^{t}\mathbf{\mathbf{V}}_{\tau}^{s}=\mathbf{\mathbf{V}}_{\tau}^{s}\mathbf{\mathbf{U}}_{s}^{t}$
holds (Ito et al. \cite{Ito04}). $\mathbf{\mathbf{U}}_{s}^{t}$ and
$\mathbf{\mathbf{V}}_{\tau}^{s}$, however, fail to commute even for
the simplest deformations, such as planar rectilinear shear (cf. formula
\eqref{eq:simple_shear_non_commute}). This implies, for instance,
that $\mathbf{R}_{\tau}^{t}$ cannot be obtained from an incremental
computation starting from an intermediate state of the body at time
$s$. We refer to this feature of the polar rotation tensor, summarized
in \eqref{eq:no_product}, as its \emph{dynamical inconsistency} (see
Fig. \ref{fig:inconsistency}). 

\begin{figure}[H]
\begin{centering}
\includegraphics[width=0.5\textwidth]{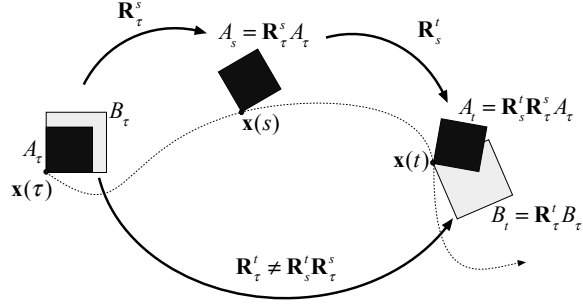}
\par\end{centering}

\caption{The action of the polar rotations $\mathbf{R}_{\tau}^{t}$, $\mathbf{R}_{\tau}^{s}$
and $\mathbf{R}_{\tau}^{t}$, illustrated on two geometric volume
elements $A_{\tau}$ and $B_{\tau}$, based at the same initial point
at time $\tau$. The evolution of $A_{\tau}$ is shown incrementally
under the subsequent polar rotations $\mathbf{R}_{\tau}^{s}$ and
$\mathbf{R}_{s}^{t}$. The evolution of the volume $B_{\tau}$ (with
initial orientation identical to that of $A_{\tau}$) is shown under
the polar rotation $\mathbf{R}_{\tau}^{t}$. All volume elements are
non-material: they only serve to illustrate how orthogonal directions
are rotated by the various polar rotations involved.}
\label{fig:inconsistency}
\end{figure}

The dynamical inconsistency of $\mathbf{R}_{\tau}^{t}$ does not imply
any flaw in the mathematics of polar decomposition. Neither does it
detract from the usefulness of $\mathbf{R}_{\tau}^{t}$ in identifying
a static rotational component of the deformation between two fixed
configurations in a geometrically optimal sense. For configurations
evolving in time, however, the polar decomposition is not an optimal
tool: the polar rotation tensor does not represent a mean material
rotation (cf. below), and the polar stretch tensor is not spin-free.
As we shall see later (cf. Section 3), both of these dynamical disadvantages
stem from the relation \eqref{eq:no_product}, which may be well-known
to experts, but is rarely, if ever, discussed in the literature. This
has led some authors to erroneously assume dynamical consistency for
$\mathbf{R}_{\tau}^{t}$ (see, e.g., Freed \cite{freed08,freed10}). 

In contrast, most textbooks in fluid mechanics introduce a mean material
rotation rate for a deforming volume element. This mean material rotation
rate is defined by Cauchy \cite{cauchy41} as the average rotation
rate of all material line elements emanating from the same point.
Cauchy's mean rotation rate turns out to be one half of the vorticity
at that point (see, e.g., Batchelor \cite{batchelor67}, Tritton \cite{tritton88}
and Vallis \cite{vallis06}). Two-dimensional experiments indeed confirm
that small, rigid objects placed in a fluid rotate at a rate that
is half of the local vorticity (Shapiro \cite{Shapiro61}). There
is, therefore, theoretical and experimental evidence for the existence
of a well-defined and observable mean material rotation in continua
that is free from the dynamical inconsistency \eqref{eq:no_product}.
Yet a connection between this mean material rotation and the finite
deformation gradient has not been established at the level of mathematical
rigor offered by the polar decomposition theorem underlying formula
\eqref{eq:PD}.

Indeed, a close link between the experimentally observed mean material
rotation in fluids (Shapiro \cite{Shapiro61}) and the rotation tensor
$\mathbf{\mathbf{R}}_{\tau}^{t}$ is only known in the limit of infinitesimally
short deformations. To state this, we need the spin tensor field $\mathbf{\mathbf{W}}(\mathbf{\mathbf{x}},t)$
and the rate-of-strain tensor field $\mathbf{D}(\mathbf{\mathbf{x}},t)$,
defined for a general velocity field $\mathbf{\mathbf{\mathbf{v}}}(\mathbf{\mathbf{x}},t)$
as
\begin{equation}
\mathbf{W}=\frac{1}{2}\left[\nabla\mathbf{v}-\left(\nabla\mathbf{v}\right)^{T}\right],\qquad\mathbf{D}=\frac{1}{2}\left[\nabla\mathbf{v}+\left(\nabla\mathbf{v}\right)^{T}\right],\label{eq:WDdef}
\end{equation}
with $\mathbf{\nabla}$ denoting the spatial gradient operation and
the superscript $T$ referring to transposition. With these ingredients,
we have the relationship 
\begin{equation}
\dot{\mathbf{R}}_{\tau}^{t}\vert_{t=\tau}=\mathbf{W(x}(t),t),\qquad\dot{\mathbf{U}}_{\tau}^{t}\vert_{t=\tau}=\dot{\mathbf{V}}_{\tau}^{t}\vert_{t=\tau}=\mathbf{D(x}(t),t),\label{eq:rotationderivative}
\end{equation}
where the dot denotes differentiation with respect to $t$ (Truesdell
\& Noll \cite{Truesdell65}). Using the definition of the vorticity
vector $\mathbf{\boldsymbol{\boldsymbol{\omega}}}=\mathbf{\nabla}\times\mathbf{\mathbf{\mathbf{v}}}$,
one therefore obtains from \eqref{eq:rotationderivative} the formula
\begin{equation}
\dot{\mathbf{R}}_{\tau}^{t}\vert_{t=\tau}\mathbf{e}=-\frac{1}{2}\mathbf{\mathbf{\boldsymbol{\omega}}}\mathbf{(x}(t),t)\times\mathbf{e},\qquad\forall\mathbf{e}\in\mathbb{R}^{3}\label{eq:vort and rot}
\end{equation}
for infinitesimally short deformations.

For deformations over a finite time interval $[\tau,t]$, the simple
relationship between the polar rotation rate and the vorticity is
lost. Only the more complex and less illuminating relationship 
\begin{equation}
\dot{\mathbf{R}}_{\tau}^{t}=\left[\mathbf{\mathbf{W(x}}(t),t)-\frac{1}{2}\mathbf{R}_{\tau}^{t}\left[\dot{\mathbf{U}}_{\tau}^{t}\left(\mathbf{U}_{\tau}^{t}\right)^{-1}-\left(\mathbf{U}_{\tau}^{t}\right)^{-1}\dot{\mathbf{U}}_{\tau}^{t}\right]\left(\mathbf{R}_{\tau}^{t}\right)^{T}\right]\mathbf{R}_{\tau}^{t}\label{eq:matrotrate}
\end{equation}
can be deduced (see, e.g., Truesdell and Rajagopal \cite{Truesdell09}).

An unexpected property of formula \eqref{eq:matrotrate}: it gives
no well-defined material rotation rate $\dot{\mathbf{\mathbf{R}}}_{\tau}^{t}\left(\mathbf{\mathbf{R}}_{\tau}^{t}\right)^{T}$
in a deforming continuum at a given location $\mathbf{x}$ and given
time $t$. Rather, the current polar rotation rate at time $t$ depends
on the starting time $\tau$ of the observation (cf. Appendix A).
This effect is not to be mistaken for the usual \emph{implicit} dependence
of kinematic tensors on the reference configuration, entering through
the dependence of the tensor on the initial conditions of its governing
differential equation. Rather, the effect arises from the \emph{explicit}
dependence of the differential equation \eqref{eq:matrotrate} on
the initial time $\tau$ through $\mathbf{U}_{\tau}^{t}$. In other
words, polar rotations do not form a dynamical system (or process):
they satisfy a nonlinear differential equation with memory (see Appendix
A for details).

Here we develop an alternative to the classic polar decomposition
which is free from these issues. Our \emph{dynamic polar decomposition}
(DPD) applies to general, non-autonomous linear processes, as opposed
to single linear operators. When applied to the deformation gradient,
the DPD yields a unique factorization $\mathbf{\mathbf{F}}_{\tau}^{t}=\mathbf{\mathbf{O}}_{\tau}^{t}\mathbf{\mathbf{M}}_{\tau}^{t}$,
with a \emph{dynamic stretch tensor} $\mathbf{\mathbf{M}}_{\tau}^{t}$
that is free from spin, and a\emph{ dynamic rotation tensor} $\mathbf{\mathbf{O}}_{\tau}^{t}$
that is free from the dynamical inconsistency \eqref{eq:no_product}.
We point out partial connections and analogies between these dynamic
tensors and prior work by Epstein \cite{epstein62}, Noll \cite{noll55}
and Rubinstein and Atluri \cite{rubinstein83} in Remark \ref{remark:prior work}
of Section 3.

The tensor $\mathbf{\mathbf{O}}_{\tau}^{t}$ is, in fact, the only
dynamically consistent rotation tensor out of the infinitely many
possible ones in \eqref{eq:infinitely_many}. Likewise, $\mathbf{\mathbf{M}}_{\tau}^{t}$
is the only spin-free stretch tensor out of the infinitely many possible
ones in \eqref{eq:infinitely_many}. Unlike the tensor pair $(\mathbf{\mathbf{R}}_{\tau}^{t},\mathbf{\mathbf{U}}_{\tau}^{t}),$
the dynamic tensor pair $(\mathbf{\mathbf{O}}_{\tau}^{t},\mathbf{\mathbf{M}}_{\tau}^{t})$
forms a dynamical system.

The dynamic rotation tensor reproduces Cauchy's mean material rotation
rate, giving the rate $\dot{\mathbf{O}}_{\tau}^{t}\left(\mathbf{O}_{\tau}^{t}\right)^{T}=\mathbf{\mathbf{\mathbf{W(x}}}(t),t\mathbf{)}$
for both finite and infinitesimal deformations. This fills the prior
mathematical gap between the deformation gradient and numerical algorithms
that rotate the reference frame incrementally (but not infinitesimally)
at the spin rate (Hughes and Winget \cite{hughes80}, Rubinstein and
Atluri \cite{rubinstein83}) rather than at the polar rotation rate.

The dynamic rotation rate $\dot{\mathbf{O}}_{\tau}^{t}\left(\mathbf{O}_{\tau}^{t}\right)^{T}$
eliminates the discrepancy of the rotation rate formula \eqref{eq:matrotrate}
with Shapiro's experiments, Helmholtz's view on continuum rotation,
and Cauchy's local mean rotation rate over all material fibers. We
also show that $\mathbf{\mathbf{O}}_{\tau}^{t}$ admits a further
factorization into a spatial mean rotation tensor and a dynamically
consistent relative rotation tensor, the latter of which is objective
for planar deformations. Finally, we introduce dynamically consistent
(i.e., temporally additive) rotation angles that extend Cauchy's classic
mean rotation, and illustrate all these concepts on two- and three-dimensional
examples.

\section{Dynamic Polar Decomposition (DPD)}

Several generalizations of the classic polar decomposition to linear
operators on various spaces are available (see, e.g., Douglas \cite{Douglas66},
Conway \cite{Conway90}) These, however, invariably target single
linear operators, as opposed to operator families, and hence exhibit
the dynamic inconsistency \eqref{eq:no_product}. 

The only polar decomposition developed specifically for time-dependent
operator families appears to be the one by Munthe--Kaas et al. \cite{Munthe01}
and Zanna and Munthe--Kaas \cite{Zanna01}). This targets Lie groups,
such as matrix-exponential solutions of linear autonomous systems
of differential equations. The decomposition, however, is approximate
and exists only for small enough $t-\tau$, i.e., for small deformations
in our context. More importantly, the deformation gradient $\mathbf{\mathbf{F}}_{\tau}^{t}$
is generally a two-parameter process (Dafermos \cite{dafermos71}),
not a one-parameter Lie group, even if the underlying deformation
field has a steady velocity field.

In order to modify the classic polar decomposition to one with dynamic
consistency, we first recall the notion of a process. We formulate
the original definition of a nonlinear process here specifically for
linear systems. The definition for nonlinear processes can be recovered
by replacing the product of two linear operators in Definition 1 with
the composition of two general functions that depend on $t$ and $\tau$
as parameters (cf. Appendix A).
\begin{defn}
A two-parameter family $\mathbf{T}_{\tau}^{t}\colon\mathbb{\mathbf{\mathbb{R}}}^{n}\to\mathbb{\mathbf{\mathbb{R}}}^{n}$,
$\tau,t\in\mathbb{\mathbf{\mathbb{R}}},$ of linear operators is a\emph{
linear process} if it is continuously differentiable with respect
to the parameters $\tau$ and $t,$ and satisfies 
\end{defn}
\[
\mathbf{T}_{t}^{t}=\mathbf{I},\qquad\mathbf{T}_{\tau}^{t}=\mathbf{T}_{s}^{t}\mathbf{T}_{\tau}^{s},\qquad\tau,s,t\in\mathbb{R}.
\]

For any linear process, we can write 
\begin{equation}
\dot{\mathbf{T}}_{\tau}^{t}=\frac{d}{d\sigma}\left.\mathbf{T}_{\tau}^{t+\sigma}\right|_{\sigma=0}=\frac{d}{d\sigma}\left.\mathbf{T}_{t}^{t+\sigma}\mathbf{T}_{\tau}^{t}\right|_{\sigma=0}=\frac{d}{d\sigma}\left.\mathbf{T}_{t}^{t+\sigma}\right|_{\sigma=0}\mathbf{T}_{\tau}^{t}.\label{eq:TODE}
\end{equation}
Therefore, any linear process $\mathbf{T}_{\tau}^{t}$ is the unique
solution of a non-autonomous linear initial value problem of the form
\begin{equation}
\mathbf{\dot{Z}}=\mathbf{A}(t)\mathbf{Z},\qquad\mathbf{Z}(\tau)=\mathbf{I};\qquad\mathbf{A}(t):=\frac{d}{d\sigma}\left.\mathbf{T}_{t}^{t+\sigma}\right|_{\sigma=0}.\label{eq:TODE-1}
\end{equation}
Conversely, the solution of any non-autonomous linear initial value
problem $\dot{\mathbf{Z}}=\mathbf{A}(t)\mathbf{Z},\,\,\,\mathbf{Z}(\tau)=\mathbf{I}$
is a linear process by the basic properties of fundamental matrix
solutions of linear differential equations (Arnold \cite{arnold78}). 
\begin{example}
\label{ex:The-deformation-gradient}The deformation gradient $\mathbf{\mathbf{F}}_{\tau}^{t}$
arising from a velocity field $\mathbf{\mathbf{v}}(\mathbf{\mathbf{x}},t)$
is a linear process, as it satisfies the equation of variations
\begin{equation}
\dot{\mathbf{Z}}=\mathbf{\nabla}\mathbf{v}(\mathbf{x}(t),t)\mathbf{Z}\label{eq:vari}
\end{equation}
with initial condition $\mathbf{Z}(\tau)=\mathbf{I}$, along the trajectory
$\mathbf{\mathbf{x}}(t)$. If the velocity field $\mathbf{\mathbf{v}}$
is irrotational ($\nabla\times\mathbf{\mathbf{v}}\equiv0$), then
its spin tensor $\mathbf{\mathbf{W}}$ vanishes, and hence $\dot{\mathbf{Z}}\mathbf{Z}^{-1}=\nabla\mathbf{\mathbf{v}}(\mathbf{\mathbf{x}}(t),t)=\mathbf{D}(\mathbf{\mathbf{x}}(t),t)$
is a symmetric matrix. Similarly, if the velocity field generates
purely rotational motion without Eulerian strain ($\mathbf{D}\equiv\mathbf{0}$),
then $\dot{\mathbf{Z}}\mathbf{Z}^{-1}=\nabla\mathbf{\mathbf{v}}(\mathbf{\mathbf{x}}(t),t)=\mathbf{\mathbf{W}}(\mathbf{\mathbf{x}}(t),t)$
is a skew-symmetric matrix. 
\end{example}
Motivated by Example \ref{ex:The-deformation-gradient}, we introduce
the following definitions for smooth, two-parameter families of operators:
\begin{defn}
\label{def:rotational}Let $\mathrm{Skew\,}(n)$, $\mathrm{Sym\,}(n)$,
and $SO(n)$ denote the set of skew-symmetric, symmetric and proper-orthogonal
linear operators on $\mathbb{\mathbf{\mathbb{R}}}^{n}$, respectively.
Also, let $\mathbf{T}_{\tau}^{t}\colon\mathbb{\mathbb{R}}^{n}\to\mathbb{\mathbb{R}}^{n}$
be a smooth, two-parameter family of linear operators. Then\end{defn}
\begin{description}
\item [{(i)}] $\mathbf{T}_{\tau}^{t}$ is \emph{rotational} if $\dot{\mathbf{T}}_{\tau}^{t}\left[\mathbf{T}_{\tau}^{t}\right]^{-1}\in\mathrm{Skew\,}(n)$
for all $\tau,t\in\mathbb{\mathbb{R}}$, or, equivalently, $\mathbf{T}_{\tau}^{t}\in SO(n)$
for all $\tau,t\in\mathbb{\mathbb{R}};$
\item [{(ii)}] $\mathbf{T}_{\tau}^{t}$ is \emph{irrotational} if $\dot{\mathbf{T}}_{\tau}^{t}\left[\mathbf{T}_{\tau}^{t}\right]^{-1}\in\mathrm{Sym\,}(n)$
for all $\tau,t\in\mathbb{\mathbb{R}}$. 
\end{description}
The equivalence of the two characterizations of time-dependent rotations
in (i) of Definition 2 is broadly known, as discussed, e.g., by Epstein
\cite{epstein66}. The concept of an irrotational linear operator
family in (ii) of Definition 2 serves as a relaxation of the concept
of symmetric operator families. Instead of requiring $\mathbf{T}_{\tau}^{t}$
to be symmetric, we only require $\dot{\mathbf{T}}_{\tau}^{t}\left[\mathbf{T}_{\tau}^{t}\right]^{-1}$
to be symmetric, which guarantees $\mathbf{T}_{\tau}^{t}$ to be the
deformation field of a purely straining linear velocity field. We
then obtain the following result on the decomposition of an arbitrary
smooth linear process $\mathbf{T}_{\tau}^{t}$ into a rotational process
and an irrotational linear transformation family.
\begin{thm}
\label{theo:DPD}{[}Dynamic Polar Decomposition (DPD){]} Any linear
process $\mathbf{T}_{\tau}^{t}\colon\mathbb{\mathbb{R}}^{n}\to\mathbb{\mathbb{R}}^{n}$
admits a unique decomposition of the form
\begin{equation}
\mathbf{T}_{\tau}^{t}=\mathbf{O}_{\tau}^{t}\mathbf{M}_{\tau}^{t}=\mathbf{N}_{\tau}^{t}\mathbf{O}_{\tau}^{t},\label{eq:DPD}
\end{equation}
where $\mathbf{\mathbf{O}}_{\tau}^{t}$ is an $n$-dimensional rotational
process, while $\mathbf{\mathbf{M}}_{\tau}^{t}$ and $\left(\mathbf{\mathbf{N}}_{\tau}^{t}\right)^{T}$
are $n$-dimensional irrotational operator families that have the
same singular values as $\mathbf{T}_{\tau}^{t}$ . Furthermore, with
the operators
\begin{equation}
\mathbf{A}^{-}(t)=\frac{1}{2}\left[\dot{\mathbf{T}}_{\tau}^{t}\mathbf{T}_{t}^{\tau}-\left(\mathbf{T}_{t}^{\tau}\right)^{T}\left(\dot{\mathbf{T}}_{\tau}^{t}\right)^{T}\right],\qquad\mathbf{A}^{+}(t)=\frac{1}{2}\left[\dot{\mathbf{T}}_{\tau}^{t}\mathbf{T}_{t}^{\tau}+\left(\mathbf{T}_{t}^{\tau}\right)^{T}\left(\dot{\mathbf{T}}_{\tau}^{t}\right)^{T}\right],\label{eq:A+-}
\end{equation}
the factors in the decomposition \eqref{eq:DPD} satisfy the linear
differential equations 
\begin{eqnarray}
\dot{\mathbf{O}}_{\tau}^{t} & = & \mathbf{A}^{-}(t)\mathbf{O}_{\tau}^{t},\qquad\qquad\qquad\mathbf{\qquad\,\,\,\,\,\,\,\,\,\,\,\,O}_{\tau}^{\tau}=\mathbf{I},\nonumber \\
\dot{\mathbf{M}}_{\tau}^{t} & = & \left[\mathbf{O}_{t}^{\tau}\mathbf{A}^{+}(t)\mathbf{O}_{\tau}^{t}\right]\mathbf{M}_{\tau}^{t},\qquad\,\,\mathbf{\qquad\,\,\,\,\,\,\,\,\,M}_{\tau}^{\tau}=\mathbf{I},\nonumber \\
\frac{d}{d\tau}\left(\mathbf{N}_{\tau}^{t}\right)^{T} & = & -\left[\mathbf{O}_{\tau}^{t}\mathbf{A}^{+}(\tau)\mathbf{O}_{t}^{\tau}\right]\left(\mathbf{N}_{\tau}^{t}\right)^{T},\qquad\left(\mathbf{N}_{t}^{t}\right)^{T}=\mathbf{I}.\label{eq:ODEs-1}
\end{eqnarray}
Both $\mathbf{A}^{-}(t)$ and $\mathbf{A}^{+}(t)$ are independent
of $\tau$, and hence $\mathbf{A}^{+}(\tau)$ is independent of $t$.\end{thm}
\begin{proof}
See Appendix B.
\end{proof}
Once the rotational process $\mathbf{O}_{\tau}^{t}$ and one of the
two irrotational operator families, $\mathbf{\mathbf{M}}_{\tau}^{t}$
and $\mathbf{\mathbf{N}}_{\tau}^{t}$, are known, the other irrotational
operator family can be computed from the relationship \eqref{eq:DPD}
.

\section{DPD of the deformation gradient}

Theorem \ref{theo:DPD} implies that the linear process $\mathbf{\mathbf{F}}_{\tau}^{t}$
(cf. Example \ref{ex:The-deformation-gradient}) can uniquely be written
as the product of left (and right) rotational and irrotational operator
families. Out of the two versions of this decomposition, the left
irrotational operator family also turns out to be objective, i.e.,
its invariants transform properly under Euclidean transformations
of the form
\begin{equation}
\mathbf{x}=\mathbf{Q}(t)\mathbf{y}+\mathbf{b}(t),\label{eq:objective}
\end{equation}
where the matrix $\mathbf{\mathbf{Q}}(t)\in SO(3)$ and the vector
$\mathbf{b}(t)\in\mathbb{R}^{3}$ are smooth functions of $t$ (Truesdell
and Noll \cite{Truesdell65}). We summarize these results in more
precise terms as follows, using notation already introduced in \eqref{eq:WDdef}.
\begin{thm}
\label{theo:FDPD}{[}DPD of the deformation gradient{]} For the deformation
field $\mathbf{X}_{t_{0}}^{t}\colon\mathcal{B}(t_{0})\subset\mathbb{R}^{3}\to\mathcal{B}(t)$,
with $t\in[t_{0},t_{1}]$, consider a trajectory $\mathbf{x}(t)$
with $\mathbf{x}(t_{0})=\mathbf{x}_{0}$. Then for any initial time
of observation $\tau\in[t_{0},t_{1}]$:

(i) The deformation gradient $\mathbf{\mathbf{F}}_{\tau}^{t}\left(\mathbf{x}(\tau)\right)$
admits a unique decomposition of the form
\begin{equation}
\mathbf{\mathbf{F}}_{\tau}^{t}=\mathbf{O}_{\tau}^{t}\mathbf{M}_{\tau}^{t}=\mathbf{N}_{\tau}^{t}\mathbf{O}_{\tau}^{t},\label{eq:DPDdef}
\end{equation}
where the dynamic rotation tensor $\mathbf{\mathbf{O}}_{\tau}^{t}$
is a rotational linear process; the dynamic right stretch tensor $\mathbf{\mathbf{M}}_{\tau}^{t}$,
as well as the transpose of the dynamic left stretch tensor $\mathbf{\mathbf{N}}_{\tau}^{t}$,
are irrotational families of transformations. 

\textup{\emph{(ii) For any $\tau,t\in\mathbb{R}$, the dynamic stretch
tensors $\mathbf{\mathbf{M}}_{\tau}^{t}$ and $\mathbf{\mathbf{N}}_{\tau}^{t}=\left(\mathbf{\mathbf{M}}_{t}^{\tau}\right)^{-1}$
are nonsingular and have the same singular values and principal axes
of strain as $\mathbf{\mathbf{U}}_{\tau}^{t}$ and $\mathbf{\mathbf{V}}_{\tau}^{t}$
do.}}

\textup{(iii) }\textup{\emph{The dynamic rotation tensor $\mathbf{\mathbf{O}}_{\tau}^{t}$,
and the dynamic stretch tensors $\mathbf{\mathbf{M}}_{\tau}^{t}$
and $\mathbf{\mathbf{N}}_{\tau}^{t}$ are solutions of the linear
initial value problems}} 
\begin{eqnarray}
\ \ \dot{\mathbf{O}}_{\tau}^{t} & = & \mathbf{W}\left(\mathbf{x}(t),t\right)\mathbf{O}_{\tau}^{t},\qquad\qquad\qquad\mathbf{\qquad\,\,\,\,\,\,\,\,\,\,\,\,\,\,\,O}_{\tau}^{\tau}=\mathbf{I},\label{eq:OODE}\\
\dot{\mathbf{M}}_{\tau}^{t} & = & \left[\mathbf{O}_{t}^{\tau}\mathbf{D}\left(\mathbf{x}(t),t\right)\mathbf{O}_{\tau}^{t}\right]\mathbf{M}_{\tau}^{t},\qquad\,\,\mathbf{\qquad\,\,\,\,\,\,\,\,\,\,\,\,\,\ M}_{\tau}^{\tau}=\mathbf{I},\label{eq:MODE}\\
\frac{d}{d\tau}\left(\mathbf{N}_{\tau}^{t}\right)^{T} & = & -\left[\mathbf{O}_{\tau}^{t}\mathbf{D}\left(\mathbf{x}(\tau),\tau\right)\mathbf{O}_{t}^{\tau}\right]\left(\mathbf{N}_{\tau}^{t}\right)^{T},\qquad\,\,\,\,\,\left(\mathbf{N}_{t}^{t}\right)^{T}=\mathbf{I}.\label{eq:NODE}
\end{eqnarray}

(iv) The left dynamic stretch tensor $\mathbf{\mathbf{N}}_{\tau}^{t}$
is objective (cf. Remark \ref{remark:objectivity}).\end{thm}
\begin{proof}
See Appendix C.
\end{proof}

\begin{rem}
A physical interpretation of the left DPD in statement (i) Theorem
2 is the following. The deformation gradient $\mathbf{\mathbf{F}}_{\tau}^{t}$
can uniquely be written as a product of two other deformation gradients:
one for a purely rotational (i.e., strainless) linear deformation
field, and one for a purely straining (i.e., irrotational) linear
deformation field. Specifically, $\mathbf{O}_{\tau}^{t}=\mathbf{\partial}_{\mathbf{a_{\tau}}}\mathbf{a}(t)$
is the deformation gradient of the strainless linear deformation $\mathbf{a}_{\tau}\mapsto\mathbf{a}(t;\tau,\mathbf{a}_{\mathbf{\tau}})$
defined by 
\begin{equation}
\dot{\mathbf{a}}=\mathbf{W}\left(\mathbf{x}(t),t\right)\mathbf{a},\label{eq:Odef}
\end{equation}
and $\mathbf{M}_{\tau}^{t}=\mathbf{\partial}_{\mathbf{b_{\tau}}}\mathbf{b}(t)$
is the deformation gradient of the irrotational linear deformation
$\mathbf{b}_{\tau}\mapsto\mathbf{b}(t;\tau,\mathbf{b}_{\tau})$ defined
by
\begin{equation}
\dot{\mathbf{b}}=\mathbf{O}_{t}^{\tau}\mathbf{D}\left(\mathbf{x}(t),t\right)\mathbf{O}_{\tau}^{t}\mathbf{b}.\label{eq:Mdef}
\end{equation}
A similar interpretation holds for the right DPD in statement (i)
of Theorem 2.
\end{rem}

\begin{rem}
Theorem \ref{theo:FDPD} guarantees that the dynamic rotation tensor
$\mathbf{\mathbf{O}}_{\tau}^{t}$ is the fundamental matrix solution
of the classical linear system of ODEs \eqref{eq:OODE}. As a consequence,
$\mathbf{\mathbf{O}}_{\tau}^{t}$ forms a linear process (or linear
dynamical system), thereby satisfying the required dynamical consistency
condition 
\begin{equation}
\mathbf{O}_{\tau}^{t}=\mathbf{O}_{s}^{t}\mathbf{O}_{\tau}^{s},\qquad\forall\,\tau,s,t\in\mathbb{R}.\label{eq:consistencyofO}
\end{equation}
By construction (cf. the proof of Theorem 2), $\mathbf{O}_{\tau}^{t}$
is the unique dynamically consistent rotation tensor out of the infinitely
many possible ones in \eqref{eq:infinitely_many}.
\end{rem}

\begin{rem}
The formula
\begin{equation}
\dot{\mathbf{U}}_{\tau}^{t}=\left[\mathbf{R}_{\tau}^{t}\right]^{T}\left[\nabla\mathbf{v}\left(\mathbf{x}(t),t\right)-\mathbf{\dot{R}}_{\tau}^{t}\left[\mathbf{R}_{\tau}^{t}\right]^{T}\right]\mathbf{U}_{\tau}^{t}\label{eq:Ueq}
\end{equation}
(see, e.g, Truesdell \& Rajagopal \cite{Truesdell09}) reveals that
$\dot{\mathbf{U}}_{\tau}^{t}\left[\mathbf{U}_{\tau}^{t}\right]^{-1}$
is not symmetric, and hence the evolution of the polar stretch tensor
is \emph{not} free from spin. Therefore, the polar decomposition does
not fully separate a purely spinning component from a non-spinning
component in the deformation. In contrast, the dynamic polar decomposition
separates a purely spinning part of the deformation gradient (cf.
\eqref{eq:Odef}) from a purely straining part with zero spin (cf.
\eqref{eq:Mdef}). By construction (cf. the proof of Theorem 2), $\mathbf{M}_{\tau}^{t}$
is the unique spin-free stretch tensor out of the infinitely many
possible ones in \eqref{eq:infinitely_many}.
\end{rem}

\begin{rem}
\label{remark:objectivity}As seen from formulas \eqref{eq:transfDPD-1}-\eqref{eq:transfDPD-2}
in the proof of Theorem 2, a general observer change \eqref{eq:objective}
transforms the dynamic rotation and stretch tensors to the form 
\[
\mathbf{\tilde{O}}_{\tau}^{t}=\mathbf{\mathbf{Q}}^{T}(t)\mathbf{O}_{\tau}^{t}\mathbf{\mathbf{Q}}(\tau),\qquad\mathbf{\mathbf{\tilde{M}}}_{\tau}^{t}=\mathbf{\mathbf{Q}}^{T}(\tau)\mathbf{M}_{\tau}^{t}\mathbf{\mathbf{Q}}(\tau),\qquad\mathbf{\mathbf{\tilde{N}}}_{\tau}^{t}=\mathbf{\mathbf{Q}}^{T}(t)\mathbf{N}_{\tau}^{t}\mathbf{\mathbf{Q}}(t)
\]
in the $\mathbf{y}$ coordinate frame. Thus, the left stretch tensor
$\mathbf{\mathbf{N}}_{\tau}^{t}$ is objective but the right stretch
tensor $\mathbf{\mathbf{M}}_{\tau}^{t}$ is not. Analogously, the
left polar stretch tensor $\mathbf{\mathbf{V}}_{\tau}^{t}$ is objective
but the right polar stretch tensor $\mathbf{\mathbf{U}}_{\tau}^{t}$
is not (cf. Truesdell \& Rajagopal \cite{Truesdell09}).
\end{rem}

\begin{rem}
The relationship \eqref{eq:DPD} gives
\[
\mathbf{M}_{\tau}^{t}=\mathbf{O}_{t}^{\tau}\mathbf{N}_{\tau}^{t}\mathbf{O}_{\tau}^{t}=\left[\mathbf{O}_{\tau}^{t}\right]^{-1}\mathbf{N}_{\tau}^{t}\mathbf{O}_{\tau}^{t}=\left[\mathbf{O}_{\tau}^{t}\right]^{T}\mathbf{N}_{\tau}^{t}\mathbf{O}_{\tau}^{t},
\]
revealing that the right dynamic stretch tensor is just the representation
of the left dynamic stretch tensor in a coordinate frame rotating
under the action of $\mathbf{O}_{\tau}^{t}$. Similarly, eq. \eqref{eq:MODE}
shows that the stretch rate tensor $\mathbf{\dot{M}}_{\tau}^{t}\left[\mathbf{M}_{\tau}^{t}\right]^{-1}$
is just the rate of strain tensor $\mathbf{D}$ represented in the
same rotating frame.
\end{rem}

\begin{rem}
The stretch tensors $\mathbf{\mathbf{M}}_{\tau}^{t}$ and $\mathbf{\mathbf{N}}_{\tau}^{t}$
are also fundamental matrix solutions, yet $\mathbf{\mathbf{M}}_{\tau}^{t}\neq\mathbf{\mathbf{M}}_{s}^{t}\mathbf{\mathbf{M}}_{\tau}^{s}$
and $\mathbf{\mathbf{N}}_{\tau}^{t}\neq\mathbf{\mathbf{N}}_{s}^{t}\mathbf{\mathbf{N}}_{\tau}^{s}.$
This is because the linear systems of ODEs \eqref{eq:MODE}-\eqref{eq:NODE}
are not of the classical type: they have right-hand sides depending
explicitly on the initial time $\tau$ as well. As a consequence,
their fundamental matrix solutions do not form processes. However,
the nonlinear system of differential equations \eqref{eq:OODE}-\eqref{eq:MODE}
has no explicit dependence on $\tau$ when posed for the dependent
variable $\mathbf{H}_{\tau}^{t}=(\mathbf{O}_{\tau}^{t},\mathbf{\mathbf{M}}_{\tau}^{t})$.
As a consequence, the nonlinear process property
\[
\mathbf{H}_{\tau}^{t}=\mathbf{H}_{s}^{t}\circ\mathbf{H}_{\tau}^{s}
\]
holds for this system of equations, and hence the pair $(\mathbf{O}_{\tau}^{t},\mathbf{\mathbf{M}}_{\tau}^{t})$
forms a nonlinear dynamical system. This is not the case for the polar
rotation-stretch pair $(\mathbf{R}_{\tau}^{t},\mathbf{\mathbf{U}}_{\tau}^{t})$
(cf. Appendix A). 
\end{rem}

\begin{rem}
The DPD of the deformation gradient in Theorem \ref{theo:FDPD} replaces
the requirement of symmetry for the polar stretch tensors $\mathbf{\mathbf{U}}_{\tau}^{t}$
and $\mathbf{\mathbf{V}}_{\tau}^{t}$ with the requirement that the
dynamic stretch tensors be deformations generated by purely straining
velocity fields. As noted in statement (ii) of Theorem \ref{theo:FDPD},
$\mathbf{\mathbf{M}}_{\tau}^{t}$ and $\mathbf{\mathbf{N}}_{\tau}^{t}$
still have the same singular values and corresponding principal axes
of strain as their polar equivalents. Thus, they continue to capture
the same objective information about stretch encoded in the right
and left Cauchy-Green strain tensors, $\mathbf{\mathbf{C}}_{\tau}^{t}=\left(\mathbf{\mathbf{F}}_{\tau}^{t}\right)^{T}\mathbf{\mathbf{F}}_{\tau}^{t}$
and $\mathbf{B}_{\tau}^{t}=\mathbf{\mathbf{F}}_{\tau}^{t}\left(\mathbf{\mathbf{F}}_{\tau}^{t}\right)^{T}$.
\end{rem}

\begin{rem}
{[}\emph{Connections with prior work}{]}\label{remark:prior work}
Without the claim of uniqueness, the first dynamic decomposition $\mathbf{\mathbf{F}}_{\tau}^{t}=\mathbf{O}_{\tau}^{t}\mathbf{M}_{\tau}^{t}$
in \eqref{eq:DPDdef} and the two equations \eqref{eq:OODE}-\eqref{eq:MODE}
could also be obtained by first extending a technical result (Theorem
1 of Epstein \cite{epstein62}) on linear differential equations to
arbitrary initial times $\tau$, and then applying this extension
to the equation of variations \eqref{eq:vari}. Also, the finite rotation
family generated by eq. \eqref{eq:OODE} is just the one considered
by Noll \cite{noll55} (p. 27) to derive isotropy-based invariance
condition for general class of (hygrosteric) constitutive laws. In
that context, however, $\mathbf{O}_{\tau}^{t}$ was selected in an
ad hoc fashion out of infinitely many possible rotations because of
the simplicity of its associated rotation rate $\dot{\mathbf{O}}_{\tau}^{t}\mathbf{O}_{\tau}^{t}\mathbf{=W}$.
Finally, eq. \eqref{eq:OODE} also appears formally in the work of
Rubenstein and Atluri \cite{rubinstein83} (see their eq. (41)). They,
however, propose this ODE merely as one generating a plausible rotating
frame in which to study deformation, as opposed to one deduced from
a systematic decomposition of the deformation gradient.
\end{rem}

\section{The relative rotation tensor }

The left dynamic stretch tensor obtained from the DPD of the deformation
gradient are objective, but the dynamic rotation tensor is not. This
is due to the inherent dependence of rigid body rotation on the reference
frame. For deforming bodies, however, there is a non-vanishing part
of the dynamic rotation that deviates from the spatial mean rotation
of the body. This relative rotation is not only dynamically consistent,
but also turns out to be objective for planar deformations.

To state this result more formally for a deforming body $\mathcal{B}(t)=\mathbf{X}_{\tau}^{t}(\mathcal{B}(\tau))$,
we denote the spatial mean of any quantity $\left(\,\cdot\,\right)(\mathbf{x},t)$
defined on $\mathcal{B}(t)$ by
\[
\overline{\left(\,\cdot\,\right)}(t)=\frac{1}{\mathrm{vol}\,(\mathcal{B}(t))}\int_{\mathcal{B}(t)}\left(\,\cdot\,\right)(\mathbf{x},t)\,dV,
\]
where $\mathrm{vol}\,(\,\,\,)$ denotes the volume for three-dimensional
bodies, and the area for two-dimensional bodies. Accordingly, $dV$
refers to the volume or area element, respectively. 
\begin{thm}
\label{theo:relativerot}{[}Relative and mean rotation tensors{]}\end{thm}
\begin{description}
\item [{\emph{(i)}}] \emph{The dynamic rotation tensor $\mathbf{\mathbf{O}}_{\tau}^{t}$
admits a unique decomposition of the form 
\begin{equation}
\mathbf{O}_{\tau}^{t}=\mathbf{\mathbf{\boldsymbol{\Phi}}}_{\tau}^{t}\mathbf{\boldsymbol{\Theta}}_{\tau}^{t}=\mathbf{\boldsymbol{\Sigma}}_{\tau}^{t}\mathbf{\mathbf{\boldsymbol{\Phi}}}_{\tau}^{t},\label{eq:object_decomp}
\end{equation}
where the relative rotation tensor $\mathbf{\mathbf{\boldsymbol{\Phi}}}_{\tau}^{t}$
and the mean rotation tensors $\mathbf{\mathbf{\boldsymbol{\Theta}}}_{\tau}^{t}$
and $\mathbf{\mathbf{\boldsymbol{\Sigma}}}_{\tau}^{t}$ satisfy the
initial value problems
\begin{eqnarray}
\dot{\mathbf{\mathbf{\boldsymbol{\Phi}}}}_{\tau}^{t} & = & \left[\mathbf{W}\left(x(t),t\right)-\bar{\mathbf{W}}\left(t\right)\right]\mathbf{\mathbf{\boldsymbol{\Phi}}}_{\tau}^{t},\qquad\mathbf{\mathbf{\boldsymbol{\Phi}}}_{\tau}^{\tau}=\mathbf{I},\label{eq:PHIODE}\\
\dot{\mathbf{\boldsymbol{\Theta}}}_{\tau}^{t} & = & \left[\mathbf{\mathbf{\boldsymbol{\Phi}}}_{t}^{\tau}\bar{\mathbf{W}}\left(t\right)\mathbf{\mathbf{\boldsymbol{\Phi}}}_{\tau}^{t}\right]\mathbf{\boldsymbol{\Theta}}_{\tau}^{t},\qquad\mathbf{\qquad\boldsymbol{\,\,\,\,\,\Theta}}_{\tau}^{\tau}=\mathbf{I},\label{eq:THETAODE}\\
\frac{d}{d\tau}\left(\mathbf{\boldsymbol{\Sigma}}_{\tau}^{t}\right)^{T} & = & \left[\mathbf{\mathbf{\boldsymbol{\Phi}}}_{\tau}^{t}\bar{\mathbf{W}}(\tau)\mathbf{\mathbf{\boldsymbol{\Phi}}}_{t}^{\tau}\right]\left(\mathbf{\mathbf{\boldsymbol{\Sigma}}}_{\tau}^{t}\right)^{T},\qquad\mathbf{\mathbf{\boldsymbol{\,\,\,\,\,\,\,\,\,\Sigma}}}_{t}^{t}=\mathbf{I}.\label{eq:SIGMAODE}
\end{eqnarray}
}
\item [{\emph{(ii)}}] \emph{The relative rotation tensor $\mathbf{\mathbf{\boldsymbol{\Phi}}}_{\tau}^{t}$
is a rotational process. For two-dimensional deformations, $\mathbf{\mathbf{\boldsymbol{\Phi}}}_{\tau}^{t}$
is also objective.}\end{description}
\begin{proof}
See Appendix D.
\end{proof}

The joint application of Theorems \ref{theo:FDPD} and \ref{theo:relativerot}
gives four possible decompositions of the deformation gradient:
\[
\mathbf{\mathbf{F}}_{\tau}^{t}=\mathbf{\mathbf{\boldsymbol{\Phi}}}_{\tau}^{t}\mathbf{\boldsymbol{\Theta}}_{\tau}^{t}\mathbf{M}_{\tau}^{t}=\mathbf{\boldsymbol{\Sigma}}_{\tau}^{t}\mathbf{\mathbf{\boldsymbol{\Phi}}}_{\tau}^{t}\mathbf{M}_{\tau}^{t}=\mathbf{N}_{\tau}^{t}\mathbf{\mathbf{\boldsymbol{\Phi}}}_{\tau}^{t}\mathbf{\boldsymbol{\Theta}}_{\tau}^{t}=\mathbf{N}_{\tau}^{t}\mathbf{\boldsymbol{\Sigma}}_{\tau}^{t}\mathbf{\mathbf{\boldsymbol{\Phi}}}_{\tau}^{t}.
\]
The relative rotation tensor $\mathbf{\mathbf{\boldsymbol{\Phi}}}_{\tau}^{t}$
is dynamically consistent and objective in two dimensions; the right
and left mean rotation tensors, $\mathbf{\mathbf{\boldsymbol{\Theta}}}_{\tau}^{t}$
and $\mathbf{\mathbf{\boldsymbol{\Sigma}}}_{\tau}^{t}$, are frame-dependent
rotational operator families. While the relative rotation tensor $\mathbf{\mathbf{\boldsymbol{\Phi}}}_{\tau}^{t}$
is generally not objective for three-dimensional deformations, it
still remains frame-invariant under all rotations $\mathbf{\mathbf{Q}}(t)$
whose rotation-rate tensor $\dot{\mathbf{\mathbf{Q}}}^{T}(t)\mathbf{\mathbf{Q}}(t)$
commutes with $\mathbf{\mathbf{\boldsymbol{\Phi}}}_{\tau}^{t}$(cf.
formula \eqref{eq:finobj2} of Appendix D).
\begin{rem}
From equations \eqref{eq:At0}, \eqref{eq:Aform1}, \eqref{eq:Aform2}
and \eqref{eq:inverse-1} of Appendix D, we deduce the following transformation
formulas for the rotation tensors featured in Theorem 3, under observer
changes of the form \eqref{eq:objective}:
\[
\tilde{\mathbf{\mathbf{\boldsymbol{\Phi}}}}_{\tau}^{t}=\mathbf{Q}^{T}(t)\mathbf{\mathbf{\boldsymbol{\Phi}}}_{\tau}^{t}\mathbf{P}(t),\qquad\tilde{\mathbf{\boldsymbol{\Theta}}}_{\tau}^{t}=\mathbf{P}^{T}(t)\mathbf{\boldsymbol{\Theta}}_{\tau}^{t}\mathbf{Q}(\tau),\qquad\mathbf{\boldsymbol{\tilde{\Sigma}}}_{\tau}^{t}=\mathbf{Q}^{T}(t)\mathbf{\boldsymbol{\Sigma}}_{\tau}^{t}\mathbf{Q}(\tau).
\]
Here the rotation tensor $\mathbf{P}(t)\in SO(3)$ satisfies the linear
initial value problem
\[
\dot{\mathbf{P}}(t)=\mathbf{\mathbf{\boldsymbol{\Phi}}}_{t}^{\tau}\dot{\mathbf{Q}}(t)\mathbf{Q}^{T}(t)\mathbf{\mathbf{\boldsymbol{\Phi}}}_{\tau}^{t}\mathbf{P}(t),\qquad\mathbf{P}(\tau)=\mathbf{Q}(\tau).
\]

\end{rem}

\section{Dynamically consistent angular velocity and mean rotation angles}

\subsection{Angular velocity from the dynamic rotation tensor\label{sub:Angular-velocity}}

By equations \eqref{eq:rotationderivative} and \eqref{eq:OODE},
the time-derivatives of the rotation tensor and the dynamic rotation
tensor agree in the limit of infinitesimally short deformations: 
\[
\dot{\mathbf{R}}_{\tau}^{t}\vert_{t=\tau}=\dot{\mathbf{O}}_{\tau}^{t}\vert_{t=\tau}=\mathbf{W}.
\]
As noted in the Introduction, however, the polar rotation does not
give a well-defined, history-independent angular velocity for finite
deformations. At the same time, the dynamic rotation gives the same
angular velocity (deduced from $\mathbf{W}$) both for infinitely
short and for finite deformations. This angular velocity equals the
mean rotation rate of material fibers in two dimensions (Cauchy \cite{cauchy41}).
Here we show that the same equality holds for three-dimensional deformations
as well.

Clearly, the rotation of an infinitesimal rigid sphere in a fluid
differs from the rotation of infinitesimal material fibers in the
fluid. Each such material fiber rotates with a different angular velocity,
even in the simplest two-dimensional steady flows (see Examples 2-3
below) Nevertheless, for all two-dimensional deformations, Cauchy
\cite{cauchy41} found that averaging the angular velocity over all
material fibers emanating from the same point gives a mean angular
velocity equal to $\frac{1}{2}\mathbf{\boldsymbol{\omega}}$ (see
also Truesdell \cite{truesdell54}). This justifies the use of small
spherical tracers to infer the rate of local mean material rotation
in two-dimensional continuum motion (see, e.g., the experiments of
Shapiro \cite{Shapiro61} for fluids). 

In three-dimensional continuum motion, the Maxey-Riley equations (Maxey
\cite{Maxey90}) continue to predict $\frac{1}{2}\mathbf{\boldsymbol{\omega}}$
as the angular velocity of small spherical particles. Experiments
on three-dimensional turbulence confirm this result (see, e.g., Meyer
et al. \cite{meyer13}). One would ideally need, however, an extension
of Cauchy's fiber-averaged angular velocity argument from two to three-dimensions
to justify equating the observed rotation rate of small rigid spheres
with the local mean rate of material rotation. 

The main challenge for such an extension is that a one-dimensional
material element has no well-defined angular velocity in three dimensions.
To see this, we let 
\begin{equation}
\mathbf{e}(t)=\frac{\mathbf{\mathbf{F}}_{\tau}^{t}\mathbf{e}(\tau)}{\left|\mathbf{\mathbf{F}}_{\tau}^{t}\mathbf{e}(\tau)\right|},\label{eq:def_e(t)}
\end{equation}
denote a unit vector tangent to a deforming material fiber along the
trajectory $\mathbf{\mathbf{x}}(t)$. This trajectory starts from
the point $\mathbf{\mathbf{x}}_{\tau}$ at time $\tau$, as shown
in Fig. \ref{fig:fiber}. 
\begin{figure}[H]
\begin{centering}
\includegraphics[width=0.5\textwidth]{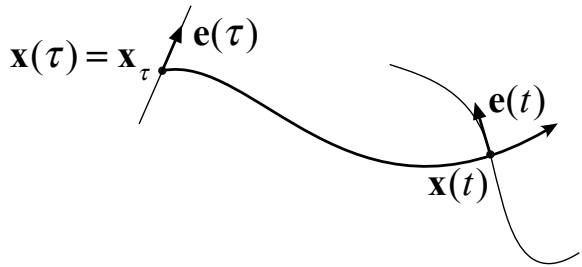}
\par\end{centering}

\caption{The unit vector $\mathbf{e}(t)$ tangent to a material fiber evolving
along the trajectory $\mathbf{\mathbf{x}}(t)$. }
\label{fig:fiber}
\end{figure}
There exists then an open half-plane $\mathcal{P}$ spanned by admissible
angular velocity vectors $\mbox{\ensuremath{\boldsymbol{\nu}}}$ such
that the instantaneous velocity $\dot{\mathbf{e}}$ of the evolving
$\mathbf{e}(t)$ satisfies $\mathbf{\dot{e}}=\mathbf{\boldsymbol{\nu}}\times\mathbf{e}$.
The magnitudes of these admissible angular velocity vectors range
from $\left|\dot{\mathbf{e}}\right|$ to infinity, depending on the
angle they enclose with $\mathbf{e}$ (see Fig. \ref{fig:plane}).
There is, therefore, no unique angular velocity for the evolving material
fiber tangent to $\mathbf{e}.$

\begin{figure}[H]
\begin{centering}
\includegraphics[width=0.55\textwidth]{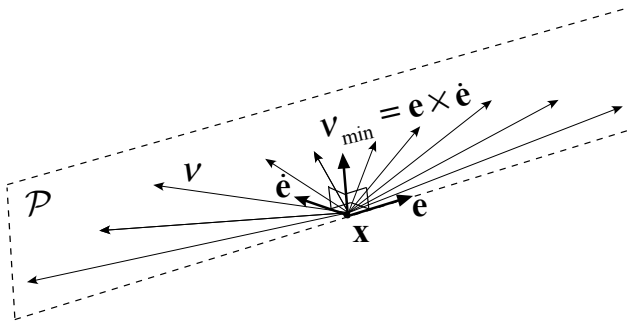}
\par\end{centering}

\caption{The plane $\mathcal{P}$ of all admissible angular velocities $\mathbf{\boldsymbol{\nu}}$
that generate the same velocity $\dot{\mathbf{e}}$ for a unit vector
$\mathbf{e}$ tangent to a deforming material element at $\mathbf{\mathbf{x}}$. }
\label{fig:plane}

\end{figure}

We can nevertheless extend Cauchy's mean rotation result to three-dimension
using the following construct. Let us define the \emph{minimal angular
velocity} vector $\mathbf{\boldsymbol{\nu}}_{\min}(\mathbf{\mathbf{x}},t,\mathbf{e})$
for the unit vector $\mathbf{e}$ as the admissible angular velocity
in $\mathcal{P}$ with the smallest possible norm: 
\begin{equation}
\mathbf{\boldsymbol{\nu}}_{\min}(\mathbf{x},t,\mathbf{e})=\mathbf{e}\times\dot{\mathbf{e}.}\label{eq:nu_min}
\end{equation}
 We then define the \emph{material-fiber-averaged angular velocity}
$\mathbf{\boldsymbol{\nu}}(\mathbf{\mathbf{x}},t)$ at the point $\mathbf{\mathbf{x}}$
of a deforming body $\mathcal{B}(t)$ by the formula
\begin{equation}
\mathbf{\boldsymbol{\nu}}(x,t):=2\left\langle \mathbf{\boldsymbol{\nu}}_{\min}(\mathbf{x},t,\mathbf{e})\right\rangle _{\mathbf{e}\in S_{\mathbf{x}}^{2}},\qquad\mathbf{x}\in\mathcal{B}(t),\label{eq:nugeneral}
\end{equation}
with the $\left\langle \,\cdot\,\right\rangle _{\mathbf{e}\in S_{\mathbf{x}}^{2}}$
operation referring to the mean over all vectors in the unit sphere
$S_{\mathbf{\mathbf{x}}}^{2}$ centered at the point $\mathbf{x}$.
For a perfectly rigid body, we recover from formula \eqref{eq:nugeneral}
the unique angular velocity of the body as the fiber-averaged angular
velocity (see Appendix E). For a general deformable continuum, $\mbox{\ensuremath{\boldsymbol{\nu}}}(\mathbf{\mathbf{x}},t)$
still turns out to be computable and equal to half of the vorticity.
\begin{prop}
\label{prop:Fiber averaged 3D}{[}Fiber-averaged angular velocity
in 3D{]} For a general three-dimensional deforming body $\mathcal{B}(t)$,
the material-fiber-averaged angular velocity at a location $\mathbf{\mathbf{x}}\in\mathcal{B}(t)$
at time $t$ is given by 
\[
\mathbf{\boldsymbol{\nu}}(\mathbf{x},t)\equiv\frac{1}{2}\mathbf{\boldsymbol{\omega}}(\mathbf{x},t),
\]
 where $\mathbf{\boldsymbol{\omega}}(\mathbf{\mathbf{x}},t)$ denotes
the vorticity vector field of $\mathcal{B}(t)$ .\end{prop}
\begin{proof}
See Appendix F.
\end{proof}
Proposition \ref{prop:Fiber averaged 3D} extends Cauchy's mean material
rotation rate result to three dimensions. It supports the expectation
that a self-consistent description of mean material rotation should
yield an instantaneous angular velocity equal to half of the vorticity
for any finite deformation, just as the dynamic rotation tensor $\mathbf{O}_{\tau}^{t}$
does.

\subsection{Dynamically consistent mean rotation angles\label{sub:Dynamically-consistent-material}}

Cauchy \cite{cauchy41} measures the magnitude of finite continuum
rotation locally by computing the rotation angle of initially co-planar
line elements about the normal of their initial plane. This \emph{mean
rotation angle} obeys a complicated, coordinate-dependent formula
(Truesdell \cite{truesdell54}) that remained unevaluated and largely
unused for a long time.

Remarkably, Zheng \& Hwang \cite{zheng92} and Huang et al. \cite{Huang96}
succeeded in evaluating the integral in Cauchy's mean rotation angle
for general planes, obtaining involved expressions defined on different
angular domains. As an alternative measure of mean rotation, Novozhilov
\cite{novozhilov71} proposed to evaluate the spatial mean of the
tangent of Cauchy's mean rotation angle, as opposed to the mean of
the angle itself, over all initially co-planar material vectors. Invariant
formulations of this idea appeared later in Truesdell \& Toupin \cite{truesdell60}
and de Oliviera et al. \cite{oliveira05}. While simpler to evaluate,
Novozhilov's version of the mean rotation angle suffers from singularities
due to the use of the tangent function (de Oliviera et al. \cite{oliveira05}).
Finally, Marzano \cite{Marzano87} proposed the mean of the cosine
of Cauchy's angle as a measure of mean rotation.

For all these mean rotation measures, the total rotation is not well-defined
beyond a range of angles due to the inherent limitations of the inverse
trigonometric functions used in their construction. A more important
issue is, however, that even fully invariant formulations of the mean
rotation angle concept (e.g., Martins \& Podiu--Guiduigli \cite{martinz92}
and Zheng, Hwang \& Betten \cite{zheng94}) extract the rotational
component of a deformation gradient via polar decomposition between
fixed initial and finite times. As a consequence, these mean rotations
are \emph{not material}: they inherit the dynamic inconsistency \eqref{eq:no_product}
of the rotation tensor. 

When evaluated along a material trajectory $\mathbf{x}(t)$, with
$\mathbf{x}(\tau)=\mathbf{x}_{\tau},$ any smooth unit vector field\emph{
$\mathbf{g}(\mathbf{\mathbf{x}},t)$} defines a time-varying axis
$\mathbf{g}(\mathbf{\mathbf{x}}(t),t)$. For any smooth rotation family
$\mathbf{Q}(s)$ defined along $\mathbf{x}(s)$ for $s\in[\tau,t]$,
the total rotation angle $\alpha_{\tau}^{t}$ with respect to the
evolving axis $\mathbf{g}(\mathbf{\mathbf{x}}(s),s)$ is equal to
\[
\alpha_{\tau}^{t}(\mathbf{x}_{\tau};\mathbf{g})=\int_{\tau}^{t}\dot{\mathbf{\boldsymbol{q}}}(s)\cdot\mathbf{g}(\mathbf{x}(s),s)\,ds,
\]
 where the angular velocity vector $\dot{\mathbf{\boldsymbol{q}}}(s)$
of $\mathbf{Q}(s)$ is defined by the relationship 
\[
\mathbf{\dot{Q}}(s)\mathbf{Q}^{T}(s)\mathbf{e}=\dot{\mathbf{\boldsymbol{q}}}(s)\times\mathbf{e},\qquad\forall\mathbf{e}\in\mathbb{R}^{3}.
\]

In line with our definition of dynamical consistency for rotation
tensors, we say that the rotation angle $\alpha_{\tau}^{t}$ with
respect to the axis field $\mathbf{g}(\mathbf{\mathbf{x}},t)$ is
\emph{dynamically consisten}t if it is additive along trajectories.
Specifically, for all times $\tau,\sigma,t\in[t_{0},t_{1}]$, the
angle $\alpha_{\tau}^{t}$ should satisfy 
\begin{equation}
\alpha_{\tau}^{t}(\mathbf{x}_{\tau};\mathbf{g})=\alpha_{\sigma}^{t}(\mathbf{x}_{\sigma};\mathbf{g})+\alpha_{\tau}^{\sigma}(\mathbf{x}_{\tau};\mathbf{g})\label{eq:dyn_const_angle}
\end{equation}
for dynamical consistency. Note that the choice $\mathbf{Q}(t)=\mathbf{R}_{\tau}^{t}$
does not give a dynamically consistent angle by formula \eqref{eq:no_product}
(cf. Remark \ref{remark: angle argument for polar rotation} in Appendix
G). The dynamic polar decomposition, however, provides several dynamically
consistent rotation angles, some of which are even objective. We keep
the terminology used for Cauchy's angle, referring to these dynamically
consistent rotation angles as mean rotation angles. This is because
they represent single-valued, overall fits to a continuum of fiber
rotation angles in a deforming volume element.
\begin{thm}
\label{theo:dynamic angles}{[}Dynamically consistent mean rotation
angles{]} \end{thm}
\begin{description}
\item [{\emph{(i)}}] \emph{The rotation angle generated by the dynamic
rotation tensor $\mathbf{O}_{\tau}^{t}$ around the axis family $\mathbf{g}$
is given by the }\textbf{\emph{dynamic rotation}}\emph{ 
\begin{equation}
\varphi_{\tau}^{t}(\mathbf{x}_{\tau};\mathbf{g})=-\frac{1}{2}\int_{\tau}^{t}\mathbf{\boldsymbol{\omega}}(\mathbf{x}(s),s)\cdot\mathbf{g}(\mathbf{x}(s),s)\,ds,\label{eq:LAV}
\end{equation}
which is a dynamically consistent rotation angle.}
\item [{\emph{(ii)}}] \emph{The rotation angle generated by the relative
rotation tensor $\mathbf{\mathbf{\boldsymbol{\Phi}}}_{\tau}^{t}$
around the axis family $\mathbf{g}$ is given by the }\textbf{\emph{relative
dynamic rotation}}\emph{ 
\begin{equation}
\phi_{\tau}^{t}(\mathbf{x}_{\tau};\mathbf{g})=-\frac{1}{2}\int_{\tau}^{t}\left[\mathbf{\boldsymbol{\omega}}(\mathbf{x}(s),s)-\bar{\mathbf{\boldsymbol{\omega}}}(s)\right]\cdot\mathbf{g}(\mathbf{x}(s),s)\,ds,\label{eq:LAVD}
\end{equation}
which is an objective and dynamically consistent rotation angle.}
\item [{\emph{(iii)}}] \emph{The rotation angle generated by the relative
rotation tensor $\mathbf{\mathbf{\boldsymbol{\Phi}}}_{\tau}^{t}$
around its own axis of rotation is given by the }\textbf{\emph{intrinsic
dynamic rotation}}\emph{ 
\begin{equation}
\psi_{\tau}^{t}(\mathbf{x}_{\tau}):=\phi_{\tau}^{t}\left(\mathbf{x}_{\tau};-\frac{\mathbf{\boldsymbol{\omega}}-\bar{\mathbf{\boldsymbol{\omega}}}}{\left|\mathbf{\boldsymbol{\omega}}-\bar{\mathbf{\boldsymbol{\omega}}}\right|}\right)=\frac{1}{2}\int_{\tau}^{t}\left|\mathbf{\boldsymbol{\omega}}(\mathbf{x}(s),s)-\bar{\mathbf{\boldsymbol{\omega}}}(s)\right|\,ds\label{eq:LAVDN}
\end{equation}
which is an objective and dynamically consistent rotation angle.}
\end{description}

\begin{proof}
See Appendix G.
\end{proof}
Figure \ref{fig:angles} illustrates the geometry of the dynamically
consistent mean rotation angles described in Theorem \ref{theo:dynamic angles}.

\begin{figure}[p]
\begin{centering}
\includegraphics[width=0.6\textwidth]{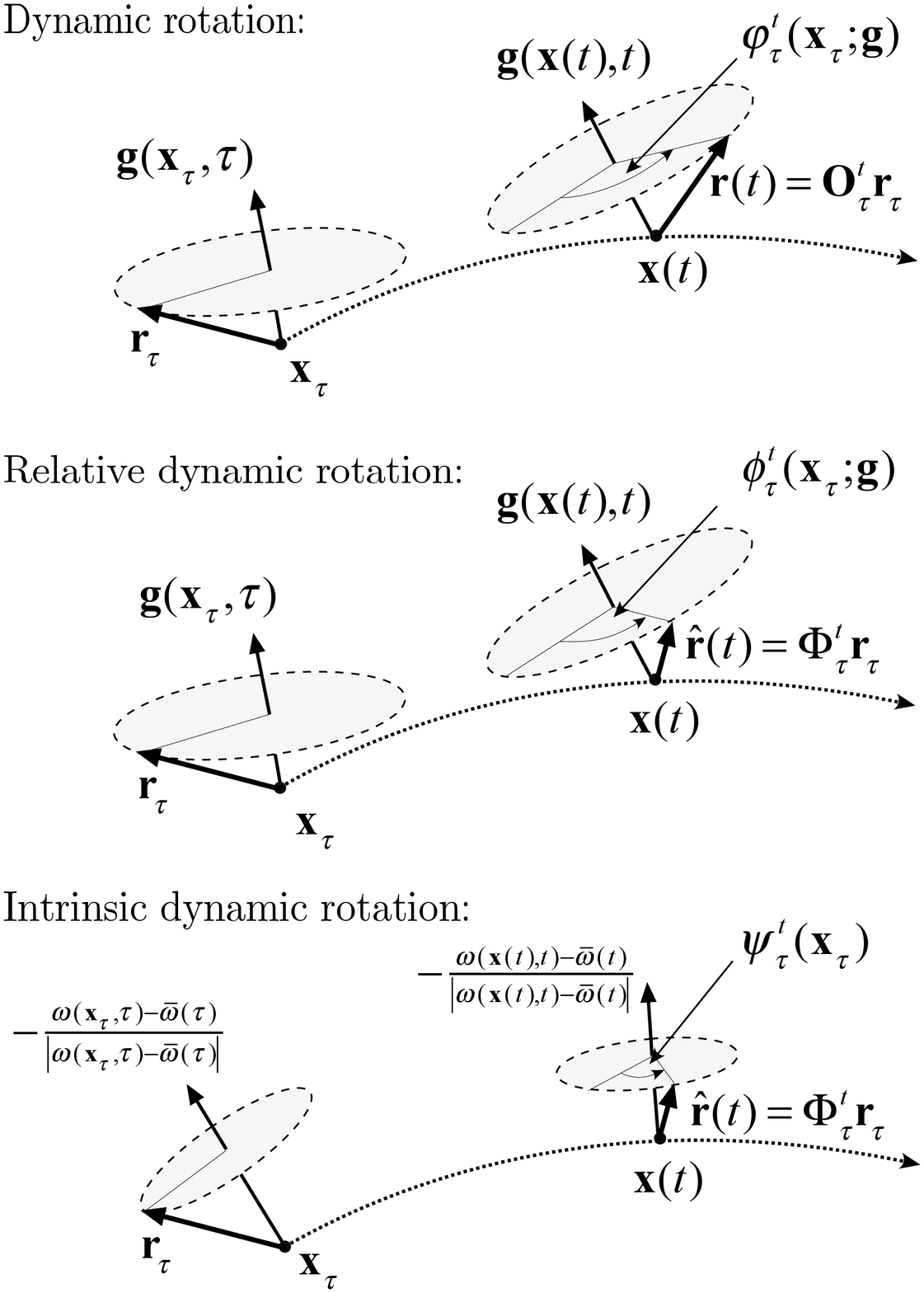}
\par\end{centering}

\caption{The geometry of the dynamic rotation, relative dynamic rotation and
intrinsic dynamic rotation obtained in Theorem \ref{theo:dynamic angles}.
Top: A vector $\mathbf{r}_{\tau}$, based at the initial point $\mathbf{x}_{\tau}$
is rotated by the dynamic rotation tensor $\mathbf{O}_{\tau}^{t}$
into the vector $\mathbf{r}(t)$, spanning the dynamic rotation angle
$\varphi_{\tau}^{t}(\mathbf{x}_{\tau};\mathbf{g})$ around an a priori
defined rotation axis family $\mathbf{g}$. Middle: The same initial
vector $\mathbf{r}_{\tau}$ is now rotated by the relative rotation
tensor $\mathbf{\Phi}_{\tau}^{t}$ into the vector $\hat{\mathbf{r}}(t)$,
spanning the relative dynamic rotation angle $\phi_{\tau}^{t}(\mathbf{x}_{\tau};\mathbf{g})$
around the axis family family $\mathbf{g}$. Bottom: $\mathbf{r}_{\tau}$
is again rotated by the relative rotation tensor $\mathbf{\Phi}_{\tau}^{t}$
into the vector $\hat{\mathbf{r}}(t)$, spanning the intrinsic dynamic
rotation angle $\psi_{\tau}^{t}(\mathbf{x}_{\tau};\mathbf{g})$ around
the the intrinsically defined rotation axis family $-\left(\mathbf{\boldsymbol{\omega}}-\bar{\mathbf{\boldsymbol{\omega}}}\right)/\left|\mathbf{\boldsymbol{\omega}}-\bar{\mathbf{\boldsymbol{\omega}}}\right|$
.}
\label{fig:angles}
\end{figure}

\begin{rem}
The intrinsic dynamic rotation $\psi_{\tau}^{t}$ measures the full
angle swept by the relative rotation tensor along the evolving the
negative relative vorticity vector $-\left(\mathbf{\boldsymbol{\omega}}-\bar{\mathbf{\boldsymbol{\omega}}}\right).$
This scalar measure is objective, even though the relative rotation
tensor $\mathbf{\Phi}_{\tau}^{t}$ generating this angle is only objective
in two dimensions. The intrinsic dynamic rotation rate 
\begin{equation}
\dot{\psi}_{\tau}^{t}(\mathbf{x})=\frac{1}{2}\left|\mathbf{\boldsymbol{\omega}}(\mathbf{x},t)-\bar{\mathbf{\boldsymbol{\omega}}}(t)\right|\label{eq:psidot}
\end{equation}
is also objective both in two- and three dimensions (cf. formula \eqref{eq: vort diff transf}
in the proof of Theorem 4). The intrinsic dynamic rotation rate is,
therefore, a viable candidate for inclusion is rotation-rate-dependent
constitutive laws. In another context, it has already been used to
define and detect rotationally coherent Eulerian vortices objectively
in two-dimensional fluid flows (Haller et al. \cite{Haller15}).
\end{rem}

\begin{rem}
The angle $\psi_{\tau}^{t}$ is always positive: its integrand generates
a positive angular increment, even if the orientation of relative
rotation changes in time due to a zero crossing of the relative vorticity.
For instance, in the two-dimensional experiments of Shapiro \cite{Shapiro61},
$\varphi_{\tau}^{t}(\mathbf{\mathbf{x}}_{\tau};\mathbf{e}_{3})$ gives
precisely the observed net rotation of a small circular body placed
in the fluid. In contrast, $\psi_{\tau}^{t}(\mathbf{\mathbf{x}}_{\tau})$
would report the total angle swept by the circular body relative to
the total mean rotation of the fluid. Both measures are objective,
as stated in Theorem \ref{theo:dynamic angles}. 
\end{rem}

\section{Dynamic rotation and stretch in two dimensions\label{sec:DPD-for-two-dimensional}}

For the material deformation induced by a two-dimensional velocity
field $\mathbf{\mathbf{v}}=(v_{1},v_{2})^{T}$, the spin tensor is
of the form
\[
\mathbf{W}\left(\mathbf{x},t\right)=\left(\begin{array}{cc}
0 & -\frac{1}{2}\omega_{3}(\mathbf{x},t)\\
\frac{1}{2}\omega_{3}(\mathbf{x},t) & 0
\end{array}\right),
\]
where $\omega_{3}=\partial_{x_{1}}v_{2}-\partial_{x_{2}}v_{1}$ is
the plane-normal component of the vorticity. The initial value problem
\eqref{eq:OODE} can then be solved by direct integration to yield
\begin{equation}
\mathbf{O}_{\tau}^{t}=\left(\begin{array}{cr}
\cos\left[\frac{1}{2}\int_{\tau}^{t}\omega_{3}(\mathbf{x}(s),s)ds\right] & -\sin\left[\frac{1}{2}\int_{\tau}^{t}\omega_{3}(\mathbf{x}(s),s)ds\right]\\
\sin\left[\frac{1}{2}\int_{\tau}^{t}\omega_{3}(\mathbf{x}(s),\tau)ds\right] & \cos\left[\frac{1}{2}\int_{\tau}^{t}\omega_{3}(\mathbf{x}(s),s)ds\right]
\end{array}\right),\label{eq:2DOgeneral}
\end{equation}
 whereas the remaining two equations \eqref{eq:MODE}-\eqref{eq:NODE}
take the form
\begin{eqnarray}
\dot{\mathbf{M}}_{\tau}^{t} & = & \left[\mathbf{O}_{t}^{\tau}\mathbf{D}\left(\mathbf{x}(t),t\right)\mathbf{O}_{\tau}^{t}\right]\mathbf{M}_{\tau}^{t},\qquad\mathbf{\qquad\qquad\,\,M}_{\tau}^{\tau}=\mathbf{I},\nonumber \\
\frac{d}{d\tau}\left(\mathbf{N}_{\tau}^{t}\right)^{T} & = & -\left[\mathbf{O}_{\tau}^{t}\mathbf{D}\left(\mathbf{x}(\tau),\tau\right)\mathbf{O}_{t}^{\tau}\right]\left(\mathbf{N}_{\tau}^{t}\right)^{T},\qquad\,\,\left(\mathbf{N}_{t}^{t}\right)^{T}=\mathbf{I},\label{eq:MNodes}
\end{eqnarray}
generally solvable only numerically. 

If the deformation gradient $\mathbf{\mathbf{F}}_{\tau}^{t}$, however,
is explicitly known, then using the self-consistency property \eqref{eq:consistencyofO},
we obtain the solutions of \eqref{eq:MNodes} directly as 
\begin{eqnarray}
\mathbf{M}_{\tau}^{t} & = & \mathbf{O}_{t}^{\tau}\mathbf{\mathbf{F}}_{\tau}^{t}\nonumber \\
 & = & \left(\begin{array}{cr}
\cos\left[\frac{1}{2}\int_{t}^{\tau}\omega_{3}(\mathbf{x}(s),s)ds\right] & -\sin\left[\frac{1}{2}\int_{t}^{\tau}\omega_{3}(\mathbf{x}(s),s)ds\right]\\
\sin\left[\frac{1}{2}\int_{t}^{\tau}\omega_{3}(\mathbf{x}(s),s)ds\right] & \cos\left[\frac{1}{2}\int_{t}^{\tau}\omega_{3}(\mathbf{x}(s),s)ds\right]
\end{array}\right)\mathbf{\mathbf{F}}_{\tau}^{t},\nonumber \\
\label{eq:2DpolarMNgeneral}\\
\mathbf{N}_{\tau}^{t} & = & \mathbf{\mathbf{F}}_{\tau}^{t}\mathbf{O}_{t}^{\tau}\nonumber \\
 & = & \mathbf{\mathbf{F}}_{\tau}^{t}\left(\begin{array}{cr}
\cos\left[\frac{1}{2}\int_{t}^{\tau}\omega_{3}(\mathbf{x}(s),s)ds\right] & -\sin\left[\frac{1}{2}\int_{t}^{\tau}\omega_{3}(\mathbf{x}(s),s)ds\right]\\
\sin\left[\frac{1}{2}\int_{t}^{\tau}\omega_{3}(\mathbf{x}(s),s)ds\right] & \cos\left[\frac{1}{2}\int_{t}^{\tau}\omega_{3}(\mathbf{x}(s),s)ds\right]
\end{array}\right).\nonumber 
\end{eqnarray}

In the present two-dimensional context, we select the rotation axis
\textbf{g} to be the unit normal $\mathbf{e}_{3}$ to the $(x_{1},x_{2})$
plane. With this choice, the unique dynamically consistent, finite
rigid-body rotation of the deformation field can be computed from
\eqref{eq:LAV} as
\begin{equation}
\varphi_{\tau}^{t}(\mathbf{x}_{\tau};\mathbf{e}_{3})=-\frac{1}{2}\int_{\tau}^{t}\omega_{3}(\mathbf{x}(s),s)\,ds.\label{eq:dynangle2D}
\end{equation}
The two-dimensional, objective expression for the relative dynamic
rotation defined in \eqref{eq:LAVD} is 
\begin{equation}
\phi_{\tau}^{t}(\mathbf{x}_{\tau};\mathbf{e}_{3})=-\frac{1}{2}\int_{\tau}^{t}\left[\omega_{3}(\mathbf{x}(s),s)-\bar{\omega}_{3}(s)\right]\,ds,\label{eq:dynangle2Ddev}
\end{equation}
while the intrinsic dynamic rotation in 
\[
\psi_{\tau}^{t}(\mathbf{x}_{\tau})=\frac{1}{2}\int_{\tau}^{t}\left|\omega_{3}(\mathbf{x}(s),s)-\bar{\omega}_{3}(s)\right|\,ds,
\]

Below we evaluate the two-dimensional DPD formulas \eqref{eq:2DOgeneral}-\eqref{eq:2DpolarMNgeneral}
and the dynamic rotation angle on the two examples of Bertrand \cite{betrand73},
which he thought proved the inability of vorticity to characterize
material rotation rates correctly (cf. Truesdell and Rajagopal \cite{Truesdell09}).
\begin{example}
\emph{Simple planar shear.} Consider the incompressible velocity field
$\mathbf{\mathbf{v}}(\mathbf{\mathbf{x}})=\left(a(x_{2}),0\right)^{T}$
for some continuously differentiable scalar function $a(x_{2})$.
The corresponding planar shear deformation gradient is 
\begin{equation}
\mathbf{\mathbf{F}}_{0}^{t}(\mathbf{x}_{0})=\left(\begin{array}{cc}
1 & a^{\prime}(x_{20})t\\
0 & 1
\end{array}\right).\label{eq:Fplanar shear}
\end{equation}
The classic polar decomposition is generally prohibitive to calculate
in the presence of parameters, even for the simple deformation gradient
\eqref{eq:Fplanar shear}. From the calculations of Dienes \cite{dienes79},
however, we obtain
\begin{eqnarray}
\mathbf{R}_{0}^{t} & = & \left(\begin{array}{cr}
\cos\left(-\tan^{-1}\left[\frac{1}{2}a^{\prime}(x_{20})t\right]\right) & -\sin\left(-\tan^{-1}\left[\frac{1}{2}a^{\prime}(x_{20})t\right]\right)\\
\sin\left(-\tan^{-1}\left[\frac{1}{2}a^{\prime}(x_{20})t\right]\right) & \cos\left(-\tan^{-1}\left[\frac{1}{2}a^{\prime}(x_{20})t\right]\right)
\end{array}\right),\label{eq:simple_shear_polar_rot}\\
\mathbf{U}_{0}^{t} & = & \left(\begin{array}{cc}
\cos\left(\tan^{-1}\left[\frac{1}{2}a^{\prime}(x_{20})t\right]\right) & \sin\left(\tan^{-1}\left[\frac{1}{2}a^{\prime}(x_{20})t\right]\right)\\
\sin\left(\tan^{-1}\left[\frac{1}{2}a^{\prime}(x_{20})t\right]\right) & a^{\prime}(x_{20})t\sin\left(\tan^{-1}\left[\frac{1}{2}a^{\prime}(x_{20})t\right]\right)+\cos\left(\tan^{-1}\left[\frac{1}{2}a^{\prime}(x_{20})t\right]\right)
\end{array}\right),\nonumber \\
\mathbf{V}_{0}^{t} & = & \left(\begin{array}{cr}
\frac{1+\sin^{2}\left(\tan^{-1}\left[\frac{1}{2}a^{\prime}(x_{20})t\right]\right)}{\cos\left(\tan^{-1}\left[\frac{1}{2}a^{\prime}(x_{20})t\right]\right)} & \sin\left(\tan^{-1}\left[\frac{1}{2}a^{\prime}(x_{20})t\right]\right)\\
\sin\left(\tan^{-1}\left[\frac{1}{2}a^{\prime}(x_{20})t\right]\right) & \cos\left(\tan^{-1}\left[\frac{1}{2}a^{\prime}(x_{20})t\right]\right)
\end{array}\right),\nonumber 
\end{eqnarray}
 showing that the polar rotation angle $\beta(t,\tau)$ satisfies
\begin{equation}
\tan\beta(t,\tau)=-\frac{1}{2}a^{\prime}(x_{20})\left(t-\tau\right).\label{eq:betadef}
\end{equation}
The dynamic inconsistency \eqref{eq:no_product} of polar rotations
is already transparent in this simple example. Indeed, noting that
\begin{equation}
\beta(t,s)=\tan^{-1}\left[\frac{1}{2}a^{\prime}(x_{20})\left(t-s\right)\right]+\tan^{-1}\left[\frac{1}{2}a^{\prime}(x_{20})s\right]\neq\tan^{-1}\left[\frac{1}{2}a^{\prime}(x_{20})t\right],\qquad t\neq0,\label{eq:betaform}
\end{equation}
 we obtain from \eqref{eq:simple_shear_polar_rot} and \eqref{eq:betaform}that
\begin{equation}
\mathbf{R}_{s}^{t}\mathbf{R}_{0}^{s}=\left(\begin{array}{cc}
\cos\left(\beta(t,s)\right) & \sin\left(\beta(t,s)\right)\\
-\sin\left(\beta(t,s)\right) & \cos\left(\beta(t,s)\right)
\end{array}\right)\neq\mathbf{R}_{0}^{t},\qquad\forall t\neq0.\label{eq:simple_shear_non_commute}
\end{equation}

To compute the dynamic polar decomposition from Theorem \ref{theo:FDPD},
we first note that 
\[
\omega_{3}(\mathbf{x}(t))=-\partial_{x_{2}}v_{1}(x_{2}(t))\equiv-a^{\prime}(x_{20}),
\]
 and hence the entries of the rate-of-strain tensor $\mathbf{D}(\mathbf{\mathbf{x}}(t))$
satisfy 
\[
D_{11}=D_{22}=0,\qquad D_{12}(\mathbf{x}(t))=D_{21}(\mathbf{x}(t))\equiv\frac{1}{2}a^{\prime}(x_{20}).
\]
 Therefore, formulas \eqref{eq:2DOgeneral}-\eqref{eq:2DpolarMNgeneral}
give the dynamic polar decomposition factors
\begin{eqnarray*}
\mathbf{O}_{0}^{t} & = & \left(\begin{array}{cr}
\cos\left[-\frac{1}{2}a^{\prime}(x_{20})t\right] & -\sin\left[-\frac{1}{2}a^{\prime}(x_{20})t\right]\\
\sin\left[-\frac{1}{2}a^{\prime}(x_{20})t\right] & \cos\left[-\frac{1}{2}a^{\prime}(x_{20})t\right]
\end{array}\right),\\
\\
\mathbf{M}_{0}^{t} & = & \left(\begin{array}{cc}
\cos\left[\frac{1}{2}a^{\prime}(x_{20})t\right] & a^{\prime}(x_{20})t\cos\left[\frac{1}{2}a^{\prime}(x_{20})t\right]-\sin\left[\frac{1}{2}a^{\prime}(x_{20})t\right]\\
\sin\left[\frac{1}{2}a^{\prime}(x_{20})t\right] & a^{\prime}(x_{20})t\sin\left[\frac{1}{2}a^{\prime}(x_{20})t\right]+\cos\left[\frac{1}{2}a^{\prime}(x_{20})t\right]
\end{array}\right),\\
\\
\mathbf{N}_{0}^{t} & = & \left(\begin{array}{cc}
a^{\prime}(x_{20})t\sin\left[\frac{1}{2}a^{\prime}(x_{20})t\right]+\cos\left[\frac{1}{2}a^{\prime}(x_{20})t\right] & a^{\prime}(x_{2})t\cos\left[\frac{1}{2}a^{\prime}(x_{20})t\right]-\sin\left[\frac{1}{2}a^{\prime}(x_{20})t\right]\\
\sin\left[\frac{1}{2}a^{\prime}(x_{20})t\right] & \cos\left[\frac{1}{2}a^{\prime}(x_{20})t\right]
\end{array}\right).\\
\end{eqnarray*}
Finally, by formula \eqref{eq:dynangle2D}, the dynamic rotation is
simply 
\[
\varphi_{0}^{t}(\mathbf{x}_{0};\mathbf{e}_{3})=-\frac{1}{2}a^{\prime}(x_{20})t,
\]
which we plot in Fig. \ref{fig:angleplots1} for comparison with the
rotation angle generated by the rotation tensor $\mathbf{\mathbf{R}}_{0}^{t}$
as a function of time. We also show in the figure the consequence
of the lack of additivity for the polar rotation, as verified in \eqref{eq:simple_shear_non_commute}.
Indeed, computing the polar rotation angle as a superposition of finite
sub-rotations, even from its analytic formula and hence without numerical
error, will give differing results.
\begin{figure}[H]
\begin{centering}
\includegraphics[width=0.8\textwidth]{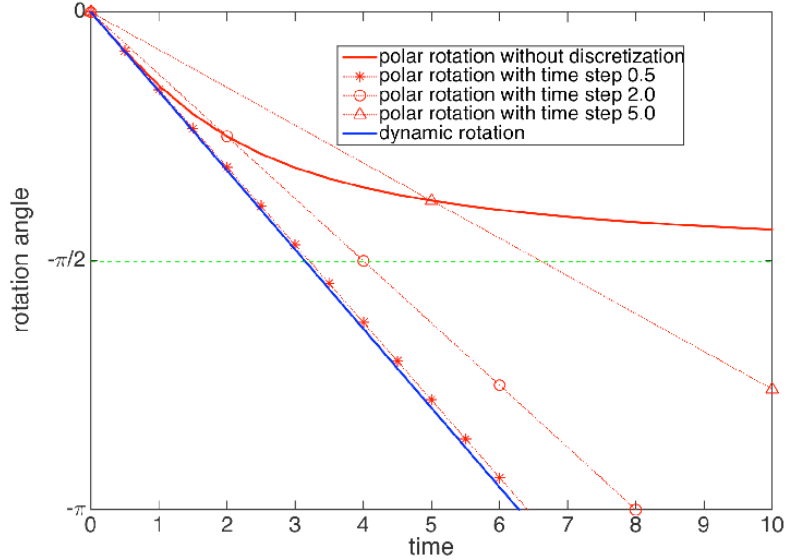}
\par\end{centering}

\caption{The classic polar rotation angle (red) and the dynamic rotation angle
(blue) as a function of time for the deformation gradient \eqref{eq:Fplanar shear}
describing planar, linear shear with $a(x_{20})=x_{20}$. Also shown
are the polar rotation angles computed from three different levels
of discretization in time. At each time step, the polar rotation angle
is incrementally recomputed, with the current time taken as the initial
time in eq. \eqref{eq:betadef}. The new rotational increment is then
added to the rotation accumulated so far. With decreasing discretization
step, the polar rotation angle computed in this incremental fashion
necessarily converges to the dynamic rotation angle by formula \eqref{eq:matrotrate}.
(Both the polar and the dynamic rotation angles represent an overall
assessment of the local rotation; individual material fibers all rotate
by different angles.)}
\label{fig:angleplots1}
\end{figure}
 By formula \eqref{eq:dynangle2Ddev}, the relative dynamic rotation
is 
\[
\phi_{0}^{t}(\mathbf{x}_{0};\mathbf{e}_{3})=-\frac{1}{2}\left[a^{\prime}(x_{20})-\overline{a^{\prime}(x_{20})}\right]t,
\]
with the overbar denoting spatial average over the domain of interest.
Finally, the intrinsic dynamic rotation is
\[
\psi_{0}^{t}(\mathbf{x}_{0})=\frac{1}{2}\left|a^{\prime}(x_{20})-\overline{a^{\prime}(x_{20})}\right|t
\]

We conclude from Fig. \ref{fig:angleplots1} that generic material
elements rotate at the well-defined mean rate $\dot{\varphi}_{0}^{t}(\mathbf{\mathbf{x}}_{0};\mathbf{e}_{3})=-\frac{1}{2}a^{\prime}(x_{20})$.
This is at odds with the polar mean rotation rate which tends to zero
over time. 

At first sight, it is the decaying polar rotation rate that agrees
with one's physical intuition. Indeed, as Flanagan and Taylor \cite{flanagan87}
write about this example: ``Clearly the body experiences rotations
which diminish over time,...''. By the end of any given finite deformation
interval, the rotation of infinitely many material fibers indeed slows
down. At the same time, however, the rotation of infinitely many other
material fibers is accelerating. For instance, at any given time $t$,
material fibers in vertical position are just reaching their maximal
material rotation rate $-\frac{3}{2}a^{\prime}(x_{20})$ (cf. formula
\eqref{eq:edotform}). Overall material fiber rotation, therefore,
does not die out.

We show a more detailed sketch of the behavior of material fibers
in Fig. \ref{fig:shear_deformation}. The frame is fixed to the trajectory
in the middle, which then becomes a set of fixed points. At any given
time, different material fibers rotate at different speeds; the lengths
of the arcs illustrate the magnitudes of angular velocities for the
corresponding material fibers. Only horizontal material fibers have
zero angular velocity. The average material angular velocity is equal
to $-\frac{1}{2}a^{\prime}(x_{20})t$ by Cauchy's classic result (Cauchy
\cite{cauchy41}) or by the restriction of our Proposition \ref{prop:Fiber averaged 3D}
to two dimensions. An infinitesimal, rigid circular tracer (shaded
area) placed in the deformation field rotates precisely at this angular
velocity. Most of this was already pointed out by Helmholtz \cite{Helmholtz68}
in his response to Bertrand \cite{betrand73}, but his observations
have apparently not been interpreted in the context of polar rotations.

\begin{figure}[H]
\begin{centering}
\includegraphics[width=0.6\textwidth]{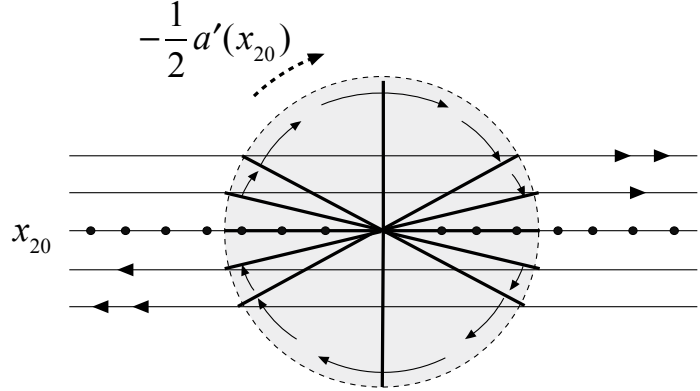}
\par\end{centering}

\caption{Rotation of material lines in a parallel shear field. }
\label{fig:shear_deformation}
\end{figure}

\end{example}

\begin{example}
\emph{Irrotational vortex}. Consider the two-dimensional, circularly
symmetric, incompressible velocity field $\mathbf{\mathbf{v}}(\mathbf{\mathbf{x}})=\left(\frac{\,-x_{2}\alpha}{x_{1}^{2}+x_{2}^{2}},\frac{\alpha x_{1}}{x_{1}^{2}+x_{2}^{2}}\right)^{T}$
, where $\alpha\in\mathbb{R}$ is a parameter. By direct calculation,
we obtain the vorticity and displacement fields 
\[
\omega_{3}\equiv0,\qquad\mathbf{X}_{0}^{t}(x_{0})=\left(\begin{array}{cc}
\cos\frac{\alpha t}{\left|\mathbf{x}_{0}\right|^{2}} & -\sin\frac{\alpha t}{\left|\mathbf{x}_{0}\right|^{2}}\\
\sin\frac{\alpha t}{\left|\mathbf{x}_{0}\right|^{2}} & \cos\frac{\alpha t}{\left|\mathbf{x}_{0}\right|^{2}}
\end{array}\right)\mathbf{x}_{0},
\]
as well as the deformation gradient
\begin{equation}
\mathbf{\mathbf{F}}_{0}^{t}(x_{0})=\left(\begin{array}{cr}
\cos\frac{\alpha t}{\left|\mathbf{x}_{0}\right|^{2}}+\frac{2x_{10}\alpha t\left(x_{10}\sin\frac{\alpha t}{\left|\mathbf{x}_{0}\right|^{2}}+x_{20}\cos\frac{\alpha t}{\left|\mathbf{x}_{0}\right|^{2}}\right)}{\left|\mathbf{x}_{0}\right|^{4}} & -\sin\frac{\alpha t}{\left|\mathbf{x}_{0}\right|^{2}}+\frac{2x_{20}\alpha t\left(x_{10}\sin\frac{\alpha t}{\left|\mathbf{x}_{0}\right|^{2}}+x_{20}\cos\frac{\alpha t}{\left|\mathbf{x}_{0}\right|^{2}}\right)}{\left|\mathbf{x}_{0}\right|^{4}}\\
\sin\frac{\alpha t}{\left|\mathbf{x}_{0}\right|^{2}}-\frac{2x_{10}\alpha t\left(x_{10}\cos\frac{\alpha t}{\left|\mathbf{x}_{0}\right|^{2}}-x_{20}\sin\frac{\alpha t}{\left|\mathbf{x}_{0}\right|^{2}}\right)}{\left|\mathbf{x}_{0}\right|^{4}} & \cos\frac{\alpha t}{\left|\mathbf{x}_{0}\right|^{2}}-\frac{2x_{20}\alpha t\left(x_{10}\cos\frac{\alpha t}{\left|\mathbf{x}_{0}\right|^{2}}-x_{20}\sin\frac{\alpha t}{\left|\mathbf{x}_{0}\right|^{2}}\right)}{\left|\mathbf{x}_{0}\right|^{4}}
\end{array}\right).\label{eq:defgrad_irrot_vortex}
\end{equation}
We also obtain from formulas \eqref{eq:2DOgeneral}-\eqref{eq:2DpolarMNgeneral}
the dynamic polar decomposition factors
\[
\mathbf{O}_{0}^{t}=\left(\begin{array}{cc}
1 & 0\\
0 & 1
\end{array}\right),\qquad\mathbf{M}_{0}^{t}=\mathbf{N}_{0}^{t}=\mathbf{\mathbf{F}}_{0}^{t}(\mathbf{x}_{0}).
\]
By formulas \eqref{eq:dynangle2D}-\eqref{eq:dynangle2Ddev}, the
dynamic rotation, the relative dynamic rotation, and the intrinsic
dynamic rotation all vanish: 
\[
\varphi_{0}^{t}(\mathbf{x}_{0};\mathbf{e}_{3})=\phi_{0}^{t}(\mathbf{x}_{0};\mathbf{e}_{3})=\psi_{0}^{t}(\mathbf{x}_{0})\equiv0.
\]
We show this together with the numerically computed polar rotation
angle in Fig. \ref{fig:angleplots2}.
\begin{figure}[H]
\begin{centering}
\includegraphics[width=0.7\textwidth]{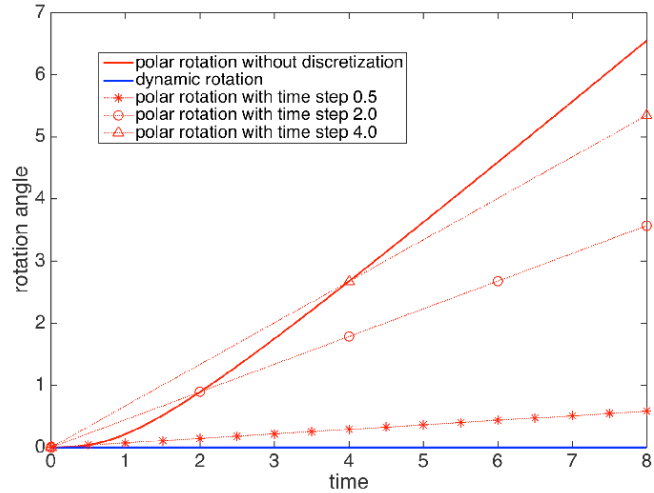}
\par\end{centering}

\caption{The classic polar rotation angle (red) and the dynamic rotation angle
(blue) as a function of time for the deformation gradient \eqref{eq:defgrad_irrot_vortex}
of an irrotational vortex with $\alpha=1$. Also shown are the exact
polar rotation angles computed incrementally for three different levels
of discretization in time, as in Fig. \ref{fig:angleplots1}. }
\label{fig:angleplots2}
\end{figure}

The vanishing dynamic rotation angle is consistent with the lack of
rotation exhibited by circular tracers in irrotational vortex experiments
(Shapiro \cite{Shapiro61}). Figure \ref{fig:irrotational_vortex}
illustrates the translation of such a tracer (shaded area). While
exceptional material fibers tangent to trajectories rotate with the
angular velocity of the trajectory, other fibers rotate in the opposite
direction due to shear. The average material angular velocity is equal
to zero by Cauchy's classic result, as well as by the restriction
of our Proposition \ref{prop:Fiber averaged 3D} to two dimensions.
Again, these observations were already made by Helmholtz \cite{Helmholtz68}
to Bertrand \cite{betrand73}, but have apparently not been evaluated
relative to the rotation predicted by the polar decomposition (see,
e.g., Dienes \cite{dienes86}, who mentions this example). 
\begin{figure}[H]
\begin{centering}
\includegraphics[width=0.5\textwidth]{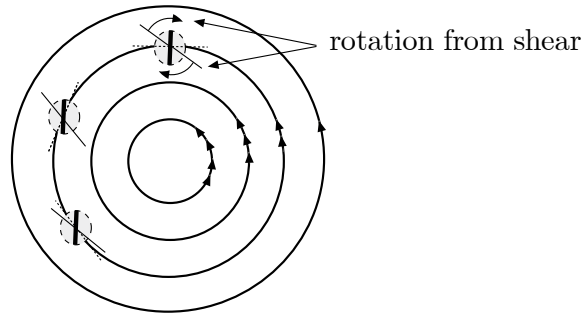}
\par\end{centering}

\caption{Rotation of material line elements around an irrotational vortex. }
\label{fig:irrotational_vortex}
\end{figure}

\end{example}

\section{Dynamic rotation and stretch in three dimensions}

For material deformation fields induced by three-dimensional velocity
fields $\mathbf{\mathbf{v}}=(v_{1},v_{2},v_{3})^{T}$, the spin tensor
can be written as
\[
\mathbf{W}\left(\mathbf{x},t\right)=\left(\begin{array}{ccc}
0 & -\frac{1}{2}\omega_{3}(\mathbf{x},t) & \frac{1}{2}\omega_{2}(\mathbf{x},t)\\
\frac{1}{2}\omega_{3}(\mathbf{x},t) & 0 & -\frac{1}{2}\omega_{1}(\mathbf{x},t)\\
-\frac{1}{2}\omega_{2}(\mathbf{x},t) & \frac{1}{2}\omega_{1}(\mathbf{x},t) & 0
\end{array}\right),
\]
where $\mathbf{\boldsymbol{\omega}}=(\omega_{1},\omega_{2},\omega_{3})=\mathbf{\nabla}\times\mathbf{\mathbf{v}}$.
The three-dimensional rotational process $\mathbf{\mathbf{O}}_{\tau}^{t}$
is the normalized fundamental matrix solution of the non-autonomous,
three--dimensional linear system of equations \eqref{eq:OODE}. At
this level of generality, \eqref{eq:OODE} must be solved numerically. 

As in the two-dimensional case, if both the rotational process $\mathbf{\mathbf{O}}_{\tau}^{t}$
and the deformation gradient $\mathbf{\mathbf{F}}_{\tau}^{t}$ are
known, then the remaining factors in the left and right DPD can be
computed as 
\begin{equation}
\mathbf{M}_{\tau}^{t}=\mathbf{O}_{t}^{\tau}\mathbf{\mathbf{F}}_{\tau}^{t},\qquad\mathbf{N}_{\tau}^{t}=\mathbf{\mathbf{F}}_{\tau}^{t}\mathbf{O}_{t}^{\tau}.\label{eq:3DMNgeneral}
\end{equation}

Finally, the dynamic rotation $\varphi_{\tau}^{t}(\mathbf{\mathbf{x}}_{\tau};\mathbf{g})$,
its relative part $\phi_{\tau}^{t}(\mathbf{\mathbf{x}}_{\tau};\mathbf{g})$
and the intrinsic rotation $\psi_{\tau}^{t}(\mathbf{\mathbf{x}}_{\tau};\mathbf{g})$
obey the formulas \eqref{eq:LAV}-\eqref{eq:LAVDN} without simplification.
\begin{example}
\emph{Three-dimensional, unsteady, parallel shear}. For a smooth,
unsteady parallel shear field in three dimensions, the velocity field
is in the general form
\[
\mathbf{v}(\mathbf{x},t)=\left(\begin{array}{c}
v_{1}(x_{3},t)\\
v_{2}(x_{3},t)\\
v_{3}(t)
\end{array}\right),
\]
where the velocity components are smooth functions of their arguments.
The spin tensor and the deformation gradient can be obtained by direct
calculation as
\begin{eqnarray}
\mathbf{W}(\mathbf{x}(t),t) & = & \left(\begin{array}{ccc}
0 & 0 & \frac{1}{2}\partial_{x_{3}}v_{1}(x_{3}(t),t)\\
0 & 0 & \frac{1}{2}\partial_{x_{3}}v_{2}(x_{3}(t),t)\\
-\frac{1}{2}\partial_{x_{3}}v_{1}(x_{3}(t),t) & -\frac{1}{2}\partial_{x_{3}}v_{2}(x_{3}(t),t) & 0
\end{array}\right),\nonumber \\
\mathbf{\mathbf{F}}_{\tau}^{t}(\mathbf{x}_{0}) & = & \left(\begin{array}{ccc}
1 & 0 & \int_{\tau}^{t}\partial_{x_{3}}v_{1}\left(x_{3}(s),s\right)ds\\
0 & 1 & \int_{\tau}^{t}\partial_{x_{3}}v_{2}\left(x_{3}(s),s\right)ds\\
0 & 0 & 1
\end{array}\right).\label{eq:DF3Dexample}
\end{eqnarray}
 The dynamic rotation tensor $\mathbf{\mathbf{O}}_{\tau}^{t}$, therefore,
satisfies the non-autonomous system of differential equations
\begin{equation}
\frac{d}{dt}\mathbf{O}_{\tau}^{t}=\left(\begin{array}{ccc}
0 & 0 & \frac{1}{2}\partial_{x_{3}}v_{1}(x_{3}(t),t)\\
0 & 0 & \frac{1}{2}\partial_{x_{3}}v_{2}(x_{3}(t),t)\\
-\frac{1}{2}\partial_{x_{3}}v_{1}(x_{3}(t),t) & -\frac{1}{2}\partial_{x_{3}}v_{2}(x_{3}(t),t) & 0
\end{array}\right)\mathbf{O}_{\tau}^{t}.\label{eq:3Dexample}
\end{equation}

Without further assumptions, this non-autonomous system can only be
solved numerically, or via an asymptotic Magnus-expansion (Magnus
\cite{magnus54}). For simplicity, we assume from now that $v_{2}(x_{3},t)\equiv cv_{1}(x_{3},t)$
for some constant $c\in\mathbb{R}$. In that case, the coefficient
matrix of \eqref{eq:3Dexample} commutes with its own integral, and
hence the fundamental matrix solution of \eqref{eq:3Dexample} is
just the exponential of the integral of its coefficient matrix (Epstein
\cite{epstein63}). Indeed, we then have
\begin{eqnarray}
\mathbf{O}_{\tau}^{t} & = & e^{\frac{1}{2}\int_{\tau}^{t}\partial_{x_{3}}v_{1}(x_{3}(\tau),\tau)d\tau}\exp\left(\begin{array}{ccc}
0 & 0 & 1\\
0 & 0 & c\\
-1 & -c & 0
\end{array}\right)\label{eq:O3Dexample}\\
 & = & \frac{\exp\left[\frac{1}{2}\int_{\tau}^{t}\partial_{x_{3}}v_{1}\left(x_{30}+\int_{\tau}^{\tau}v_{3}(\sigma)\,d\sigma,s\right)\,ds\right]}{c^{2}+1}\nonumber \\
 &  & \times\left(\begin{smallmatrix}\begin{array}{ccc}
\cos\left(c^{2}+1\right)+c^{2} & c\left[\cos\left(c^{2}+1\right)-1\right] & \sqrt{c^{2}+1}\left(\sin\left(c^{2}+1\right)+c^{2}\right)\\
c\left[\cos\left(c^{2}+1\right)-1\right] & c^{2}\cos\left(c^{2}+1\right)+1 & c\sqrt{c^{2}+1}\sin\left(c^{2}+1\right)\\
-\sqrt{c^{2}+1}\left(\sin\left(c^{2}+1\right)+c^{2}\right) & -c\sqrt{c^{2}+1}\sin\left(c^{2}+1\right) & \cos\left(c^{2}+1\right)
\end{array}\end{smallmatrix}\right).\nonumber 
\end{eqnarray}
Then, from formulas \eqref{eq:3DMNgeneral}, \eqref{eq:DF3Dexample}
and \eqref{eq:O3Dexample} we obtain the left and right DPD factors
$\mathbf{\mathbf{M}}_{\tau}^{t}$ and $\mathbf{\mathbf{N}}_{\tau}^{t}$
explicitly, which we omit here for brevity. With the vorticity vector
\[
\mathbf{\mathbf{\boldsymbol{\omega}}}=\partial_{x_{3}}v_{1}\left(\begin{array}{c}
1\\
-c\\
0
\end{array}\right),
\]
and with respect to a constant rotation axis defined by a unit vector
$\mathbf{g}=(g_{1},g_{2},g_{3})^{T}$, the frame-dependent dynamic
rotation angle is of the form 
\[
\varphi_{\tau}^{t}(\mathbf{x}_{\tau};\mathbf{g})=\frac{1}{2}\left(g_{1}-cg_{2}\right)\int_{\tau}^{t}\partial_{x_{3}}v_{1}\left(x_{30}+\int_{\tau}^{s}v_{3}(\sigma)\,d\sigma,s\right)\,ds.
\]
In contrast, the (objective) relative dynamic rotation is given by
\[
\phi_{\tau}^{t}(\mathbf{x}_{\tau};\mathbf{g})=\frac{1}{2}\left(g_{1}-cg_{2}\right)\int_{\tau}^{t}\left[\partial_{x_{3}}v_{1}\left(x_{30}+\int_{\tau}^{s}v_{3}(\sigma)\,d\sigma,s\right)-\overline{\partial_{x_{3}}v_{1}}(s)\right]\,ds,
\]
 and the (objective) intrinsic dynamic rotation is given by
\[
\psi_{\tau}^{t}(\mathbf{x}_{\tau})=\frac{1}{2}\sqrt{c^{2}+1}\int_{\tau}^{t}\left|\partial_{x_{3}}v_{1}\left(x_{30}+\int_{\tau}^{s}v_{3}(\sigma)\,d\sigma,s\right)-\overline{\partial_{x_{3}}v_{1}}(s)\right|\,ds.
\]

\end{example}

\section{Conclusions}

The classic polar decomposition of the deformation gradient is a broadly
employed tool in analyzing continuum deformation. Given the deformation
gradient, one obtains the polar rotation and stretch tensors from
algorithms based on straightforward linear algebra. Beyond computational
simplicity, polar rotation offers a powerful and rigorous tool to
identify a static rotational component of the linearized deformation
between fixed initial and final configurations.

Polar rotations computed over different time intervals, however, do
not have the fundamental additivity property of solid-body rotations.
As a consequence, polar rotation does not identify a mean material
rotation for volume elements which is nevertheless experimentally
observable in fluids (Shapiro \cite{Shapiro61}). Polar rotation also
suggests a mean angular velocity distribution that depends on the
length of the observation period, introducing an irremovable memory
effect into the deformation history on purely kinematic grounds (cf.
Appendix A). Finally, the evolution of the polar stretch tensor is
not free from spin. In summary, the static optimality of the polar
decomposition between two fixed configurations also comes with dynamic
sub-optimality for time-varying configurations.

To address these disadvantages, here we have extended the idea of
polar decomposition from a single linear mapping between two fixed
configurations to a time-dependent process. The resulting dynamic
polar decomposition (DPD) yields unique left and right factorizations
of $\mathbf{F}_{\tau}^{t}$ into the deformation gradient of a purely
rotating (strainless) deformation and the deformation gradient of
a purely straining (irrotational) deformation. The former deformation
gradient, the dynamic rotation tensor, is a dynamically consistent
rotation family. The latter deformation gradient, the (left) dynamic
stretch tensor, is objective, just as its classic polar left stretch
counterpart. The dynamic stretch tensors also reproduce the same Cauchy--Green
strains and principal strain directions between any two configurations,
as the classic polar stretch tensors do. Unlike the right polar stretch
tensor, however, the right dynamic stretch tensor is spin-free.

The DPD provides a previously missing mathematical link between the
deformation gradient and numerical algorithms that rotate the reference
frame incrementally at the spin rate (Hughes and Winget \cite{hughes80},
Rubinstein and Atluri \cite{rubinstein83}). The dynamic rotation
tensor arising from the DPD reproduces precisely the mean material
rotation rate of volume elements, as defined by Cauchy \cite{cauchy41}.
This mean rotation rate is directly observable in two-dimensional
fluids by placing a small spherical tracer in the flow (Shapiro \cite{Shapiro61}).
The same experiment cannot be carried out for solids. Any possible
experiment in solids, however, with an ability to measure the average
rotation rate of all fibers in a material volume element, necessarily
has to return the rate obtained from the DPD (cf. Proposition 1).

The DPD also provides new dynamic rotation angles for volume elements.
These angles represent dynamically consistent and simply computable
alternatives to Cauchy's classic mean rotation angle, whose evaluation
has been difficult using the classic polar decomposition (cf. Section
\ref{sub:Dynamically-consistent-material}). The dynamic rotation
angles also enable the extension of polar-rotation-based material
vortex detection in two-dimensional deformations (Farazmand and Haller
\cite{Farazmand14}) to DPD-based material vortex detection in three-dimensions
(Haller et al. \cite{Haller15}).

On the computational side, the DPD cannot be obtained from simple
linear algebraic manipulations on the single linear mapping $\mathbf{F}_{\tau}^{t}$,
as is the case for the classic polar decomposition. Instead, one has
to solve non-autonomous linear differential equations over the time
interval $[\tau,t]$ to obtain the DPD of $\mathbf{F}_{\tau}^{t}$.
On the upside, this also means that the dynamic rotation-stretch tensor
pair $(\mathbf{O}_{\tau}^{t},\mathbf{M}_{\tau}^{t})$ together satisfies
an explicit system of differential equations, i.e., form a dynamical
system that is free from memory effects. This is not the case for
their classic polar counterparts: $(\mathbf{R}_{\tau}^{t},\mathbf{U}_{\tau}^{t})$
satisfy an implicit, nonlinear system of differential equations, which
does not define a dynamical system and has unavoidable memory effects
(cf. Appendix A).

We believe that memory effects should enter models of the deformation
process in a controlled fashion, through parameters in the constitutive
equations, rather than in an uncontrolled and un-parametrized fashion,
through the rotational kinematics. For this reason, we consider the
intrinsic dynamic rotation rate $\dot{\psi}_{\tau}^{t}$, defined
in \eqref{eq:psidot}, a viable candidate for inclusion in constitutive
laws, given that it is simple, objective and memory-free. For two-dimensional
deformations, the rotation rate $\dot{\mathbf{\mathbf{\boldsymbol{\Phi}}}}_{\tau}^{t}\mathbf{\mathbf{\boldsymbol{\Phi}}}_{\tau}^{t}=\mathbf{W}-\bar{\mathbf{W}}$
of the relative rotation tensor can also be used, as it is objective
by eq. \eqref{eq:transfSOmega-3}.

Finally, we expect the DPD to be useful in experimental techniques
producing time-resolved deformation with large strains. An example
is the Digital Image Correlation (DIC) applied to granular materials,
where the classic polar rotation tensor has been used so far to identify
macroscopic rigid-body rotation components of the deformation field
(see, e.g., Rechenmacher et al. \cite{rechenmacher11}).

\vskip 0.5 true cm

\textbf{Acknowledgement} I acknowledge very helpful discussions with
Alexander Ehret, Mohammad Farazmand, Florian Huhn, Edoardo Mazza and
David Öttinger. I am also grateful for the insightful suggestions
of the two anonymous reviewers of this article.

\section{Appendix A: Polar rotations do not form a dynamical system}

We start by recalling the well-known temporal evolution of the deformation
gradient. Let us fix a material trajectory $\mathbf{x}(t)$, with
$\mathbf{x}(t_{0})=\mathbf{x}_{0}$. The deformation gradient along
this trajectory obeys the differential equation (cf. Example 1)
\begin{equation}
\dot{\mathbf{F}}_{\tau}^{t}=\mathbf{\nabla}\mathbf{v}(\mathbf{x}(t),t)\mathbf{F}_{\tau}^{t},\label{eq:FODE}
\end{equation}
where $\mathbf{v}(\mathbf{x},t)$ is the velocity field associated
with the deformation. The time $\tau\in[t_{0},t_{1}]$ is arbitrary,
labeling a reference configuration from which an observer follows
the deformation gradient up to time $t\in[t_{0},t_{1}].$ The solution
$\mathbf{F}_{\tau}^{t}$ of the differential equation \eqref{eq:FODE},
therefore, depends \emph{implicitly} on the start time $\tau$ of
the observation, without $\tau$ entering the differential equation
explicitly.

A deformation rate tensor (analogous to the polar rotation rate) can
also be defined for the deformation gradient as
\begin{equation}
\dot{\mathbf{F}}_{\tau}^{t}\left(\mathbf{F}_{\tau}^{t}\right)^{-1}=\mathbf{\nabla}\mathbf{v}(\mathbf{x}(t),t).\label{eq:defrate}
\end{equation}
This gives a well-defined deformation rate at the point $\mathbf{x}(t)$
of the deformed configuration at time $t$, independent of the initial
time $\tau$ at which the observer started monitoring the linearized
deformation along the trajectory $\mathbf{\mathbf{x}}(t)$. One may,
in particular, select the start time of the observation as $t=\tau$
and obtain the same rate $\dot{\mathbf{F}}_{t}^{t}\left(\mathbf{F}_{t}^{t}\right)^{-1}=\mathbf{\nabla}\mathbf{v}(\mathbf{\mathbf{x}}(t),t)$. 

We now show that this is not the case for the polar rotation rate.
The differential equation for the polar rotation tensor along the
trajectory $\mathbf{x}(t)$ is of the form
\begin{equation}
\dot{\mathbf{\mathbf{R}}}_{\tau}^{t}=\left[\mathbf{W}(\mathbf{x}(t),t)-\frac{1}{2}\mathbf{R}_{\tau}^{t}\left[\dot{\mathbf{U}}_{\tau}^{t}\left(\mathbf{U}_{\tau}^{t}\right)^{-1}-\left(\mathbf{U}_{\tau}^{t}\right)^{-1}\dot{\mathbf{U}}_{\tau}^{t}\right]\left(\mathbf{R}_{\tau}^{t}\right)^{T}\right]\mathbf{\mathbf{R}}_{\tau}^{t}.\label{eq:nonlinR}
\end{equation}
This gives the instantaneous rotation rate at the point $\mathbf{x}(t)$,
at time $t$, in the form 
\begin{equation}
\dot{\mathbf{\mathbf{R}}}_{\tau}^{t}\left(\mathbf{\mathbf{R}}_{\tau}^{t}\right)^{T}=\mathbf{W}(\mathbf{x}(t),t)-\frac{1}{2}\mathbf{R}_{\tau}^{t}\left[\dot{\mathbf{U}}_{\tau}^{t}\left(\mathbf{U}_{\tau}^{t}\right)^{-1}-\left(\mathbf{U}_{\tau}^{t}\right)^{-1}\dot{\mathbf{U}}_{\tau}^{t}\right]\left(\mathbf{R}_{\tau}^{t}\right)^{T}\label{eq:Rrate}
\end{equation}
for the observer monitoring the infinitesimal deformation along $\mathbf{x}(t)$
from the initial time $\tau$ up to the present time $t$. Note that
this rate depends \emph{explicitly} on the initial time $\tau$ of
observation through $\mathbf{U}_{\tau}^{t}$. In particular, for an
observation starting at time $t=\tau,$ we obtain $\dot{\mathbf{R}}_{t}^{t}\left(\mathbf{R}_{t}^{t}\right)^{-1}=\mathbf{W}(\mathbf{x}(t),t),$
which is quite different from \eqref{eq:Rrate} with $\tau\neq t.$
Therefore, the instantaneous polar rotation rate at a given location
and time is ill-defined when different start times for the observation
are allowed.

There is, in fact, a deeper effect at play here. Rather than examining
the rates $\dot{\mathbf{F}}_{\tau}^{t}\left(\mathbf{F}_{\tau}^{t}\right)^{-1}$
and $\dot{\mathbf{\mathbf{R}}}_{\tau}^{t}\left(\mathbf{\mathbf{R}}_{\tau}^{t}\right)^{T}$,
let us simply examine if the derivatives $\dot{\mathbf{F}}_{\tau}^{t}$
and $\dot{\mathbf{\mathbf{R}}}_{\tau}^{t}$ are independent of the
observational history. Note that the derivative of $\mathbf{F}_{\tau}^{t}$
in \eqref{eq:FODE} only depends on the current time $t$ and $\mathbf{F}_{\tau}^{t}$
itself. Thus, in the language of differential equations, \eqref{eq:FODE}
is a non-autonomous dynamical system (or a \emph{process}; cf. Dafermos
\cite{dafermos71}) for the deformation gradient $\mathbf{F}_{\tau}^{t}$,
with its future evolution fully determined by its present state. The
defining properties of a process, spelled out for the tensor family
$\mathbf{F}_{\tau}^{t}$, are 
\[
\mathbf{F}_{\tau}^{t}=\mathbf{F}_{s}^{t}\circ\mathbf{F}_{\tau}^{s},\qquad\mathbf{F}_{t}^{t}=\mathbf{I},\qquad\forall s,\tau,t\in[t_{0},t_{1}],
\]
with the circle denoting the composition of two functions. By the
linearity of \eqref{eq:FODE}, $\mathbf{F}_{\tau}^{t}$ is actually
a linear process, and hence we simply have $\mathbf{F}_{s}^{t}\circ\mathbf{F}_{\tau}^{s}=\mathbf{F}_{s}^{t}\mathbf{F}_{\tau}^{s}$.
The linearity of the dynamical system \eqref{eq:FODE}, however, plays
no role in our current argument.\footnote{An example of a nonlinear process is a general deformation field,
satisfying the nonlinear differential equation $\dot{\mathbf{X}}_{\tau}^{t}=\mathbf{v}(\mathbf{X}_{\tau}^{t},t).$
In this case, the nonlinear process properties take the form $\mathbf{X}_{\tau}^{t}=\mathbf{X}_{s}^{t}\circ\mathbf{X}_{\tau}^{s}$
and $\mathbf{X}_{t}^{t}=\mathbf{I,}$ for all $s,\tau,t\in[t_{0},t_{1}].$
Here the function composition cannot be replaced by a simple product.}

In contrast, the derivative of $\mathbf{R}_{\tau}^{t}$ in \eqref{eq:nonlinR}
depends on the current time $t$, on the tensor $\mathbf{R}_{\tau}^{t}$
itself, as well as on the initial time $\tau$ of the observation
through the quantity $\mathbf{U}_{\tau}^{t}$. As a consequence, the
nonlinear differential equation \eqref{eq:nonlinR} is \emph{not}
a dynamical system (or process), because its future evolution is not
determined fully by its present state, and hence 
\begin{equation}
\mathbf{R}_{\tau}^{t}\neq\mathbf{R}_{s}^{t}\circ\mathbf{R}_{\tau}^{s}\label{eq:nocomp}
\end{equation}
holds. Thus, in addition to not being a linear process by \eqref{eq:no_product},
the polar rotation tensor also fails to be a nonlinear process by
property \eqref{eq:nocomp}. Instead, $\mathbf{R}_{\tau}^{t}$ satisfies
a nonlinear differential equation with \emph{memory}. 

Even when considered together, the $(\mathbf{R}_{\tau}^{t},\mathbf{U}_{\tau}^{t})$
tensor pair does not satisfy an explicit system of differential equations.
Rather, the pair satisfies a nonlinear implicit system of differential
equations formed by \eqref{eq:matrotrate}-\eqref{eq:Ueq} (albeit
this system has no explicit dependence on $\tau$). As a consequence,
the pair $(\mathbf{R}_{\tau}^{t},\mathbf{U}_{\tau}^{t})$ generally
does not form a nonlinear dynamical system (or nonlinear process)
either, and hence displays explicit memory effects beyond the customary
implicit dependence on the reference configuration.

\section{Appendix B: Proof of Theorem \ref{theo:DPD}}

Substituting the decomposition in \eqref{eq:DPD} into \eqref{eq:A+-},
and imposing the requirement that $\mathbf{\mathbf{O}}_{\tau}^{t}$
is rotational and $\mathbf{\mathbf{M}}_{\tau}^{t}$ is irrotational
(cf. Definition \ref{def:rotational}), we obtain that 
\begin{eqnarray}
\mathbf{A}^{-}(t) & = & \dot{\mathbf{O}}_{\tau}^{t}\left[\mathbf{O}_{\tau}^{t}\right]^{T}+\frac{1}{2}\mathbf{O}_{\tau}^{t}\left[\dot{\mathbf{M}}_{\tau}^{t}\left[\mathbf{M}_{\tau}^{t}\right]^{-1}-\left[\dot{\mathbf{M}}_{\tau}^{t}\left[\mathbf{M}_{\tau}^{t}\right]^{-1}\right]^{T}\right]\left[\mathbf{O}_{\tau}^{t}\right]^{T}\nonumber \\
 & = & \dot{\mathbf{O}}_{\tau}^{t}\left[\mathbf{O}_{\tau}^{t}\right]^{T},\nonumber \\
\mathbf{A}^{+}(t) & = & \frac{1}{2}\left[\dot{\mathbf{O}}_{\tau}^{t}\left[\mathbf{O}_{\tau}^{t}\right]^{T}+\mathbf{O}_{\tau}^{t}\left[\dot{\mathbf{O}}_{\tau}^{t}\right]^{T}\right]+\frac{1}{2}\mathbf{O}_{\tau}^{t}\left[\dot{\mathbf{M}}_{\tau}^{t}\left[\mathbf{M}_{\tau}^{t}\right]^{-1}+\left[\dot{\mathbf{M}}_{\tau}^{t}\left[\mathbf{M}_{\tau}^{t}\right]^{-1}\right]^{T}\right]\left[\mathbf{O}_{\tau}^{t}\right]^{T}\nonumber \\
 & = & \frac{1}{2}\frac{d}{dt}\left[\mathbf{O}_{\tau}^{t}\left[\mathbf{O}_{\tau}^{t}\right]^{T}\right]+\frac{1}{2}\mathbf{O}_{\tau}^{t}\left[\dot{\mathbf{M}}_{\tau}^{t}\left[\mathbf{M}_{\tau}^{t}\right]^{-1}+\left[\dot{\mathbf{M}}_{\tau}^{t}\left[\mathbf{M}_{\tau}^{t}\right]^{-1}\right]^{T}\right]\left[\mathbf{O}_{\tau}^{t}\right]^{T}\nonumber \\
 & = & \frac{1}{2}\mathbf{O}_{\tau}^{t}\left[\dot{\mathbf{M}}_{\tau}^{t}\left[\mathbf{M}_{\tau}^{t}\right]^{-1}+\left[\dot{\mathbf{M}}_{\tau}^{t}\left[\mathbf{M}_{\tau}^{t}\right]^{-1}\right]^{T}\right]\left[\mathbf{O}_{\tau}^{t}\right]^{T}\nonumber \\
 & = & \mathbf{O}_{\tau}^{t}\dot{\mathbf{M}}_{\tau}^{t}\left[\mathbf{M}_{\tau}^{t}\right]^{-1}\left[\mathbf{O}_{\tau}^{t}\right]{}^{T}.\label{eq:A+-2}
\end{eqnarray}
 Expressing the derivatives of $\mathbf{\mathbf{O}}_{\tau}^{t}$ and
$\mathbf{\mathbf{M}}_{\tau}^{t}$ from \eqref{eq:A+-2} proves the
first two equations in \eqref{eq:ODEs-1}. We also note that
\[
\left[\mathbf{T}_{\tau}^{t}\right]^{T}\mathbf{T}_{\tau}^{t}=\left[\mathbf{M}_{\tau}^{t}\right]^{T}\left[\mathbf{O}_{\tau}^{t}\right]^{T}\mathbf{O}_{\tau}^{t}\mathbf{M}_{\tau}^{t}=\left[\mathbf{M}_{\tau}^{t}\right]^{T}\mathbf{M}_{\tau}^{t},
\]
and hence $\mathbf{T}_{\tau}^{t}$ and $\mathbf{\mathbf{M}}_{\tau}^{t}$
have the same singular values, as claimed.

Using the notation from eq. \eqref{eq:TODE-1} in eq. \eqref{eq:TODE},
we can further write
\[
\dot{\mathbf{T}}_{\tau}^{t}\mathbf{T}_{t}^{\tau}=\mathbf{A}(t)=\mathbf{A}^{-}(t)+\mathbf{A}^{+}(t),
\]
 with $\mathbf{A}^{\pm}(t)=\frac{1}{2}\left[\mathbf{A}(t)\pm\mathbf{A}^{T}(t)\right]$.
Therefore, $\mathbf{A}^{\pm}(t)$ are indeed independent of $\tau$
and hence $\mathbf{A}^{\pm}(\tau)$ are independent of $t,$ as already
suggested by our notation. 

From the now proven first equation of \eqref{eq:ODEs-1}, we conclude
that $\mathbf{\mathbf{O}}_{\tau}^{t}$ is indeed a linear process,
as the fundamental matrix solution of a classic non-autonomous system
of linear ODEs (with no explicit dependence on the initial time $\tau$).
We also conclude that $\mathbf{\mathbf{M}}_{\tau}^{t}$ is a two-parameter
family of nonsingular operators, even though it is generally not a
process. In particular, $\mathbf{\mathbf{M}}_{\tau}^{t}$ does not
form a process because the coefficient matrix of the second system
of ODEs in \eqref{eq:ODEs-1} has explicit dependence on the initial
time $\tau$. As a consequence, we generally have
\[
\left(\mathbf{M}_{t}^{\tau}\right)^{-1}\neq\mathbf{M}_{\tau}^{t}.
\]

To prove the left-polar decomposition involving $\mathbf{\mathbf{N}}_{\tau}^{t}$
in \eqref{eq:DPD}, we observe that 
\[
\mathbf{T}_{\tau}^{t}=\left[\mathbf{T}_{t}^{\tau}\right]^{-1}=\left[\mathbf{O}_{t}^{\tau}\mathbf{M}_{t}^{\tau}\right]^{-1}=\left(\mathbf{M}_{t}^{\tau}\right)^{-1}\left(\mathbf{O}_{t}^{\tau}\right)^{-1}=\left(\mathbf{M}_{t}^{\tau}\right)^{-1}\mathbf{O}_{\tau}^{t},
\]
thus setting 
\begin{equation}
\mathbf{N}_{\tau}^{t}=\left(\mathbf{M}_{t}^{\tau}\right)^{-1},\label{eq:inverse}
\end{equation}
we conclude the existence of $\mathbf{\mathbf{N}}_{\tau}^{t}$, as
claimed.  Interchanging the role of $\tau$ and $t$ in the second
equation of \eqref{eq:ODEs-1}, we obtain the differential equation
\begin{equation}
\frac{d}{d\tau}\mathbf{M}_{t}^{\tau}=\left[\mathbf{O}_{\tau}^{t}\mathbf{A}^{+}(\tau)\mathbf{O}_{t}^{\tau}\right]\mathbf{M}_{t}^{\tau},\qquad\mathbf{M}_{t}^{t}=\mathbf{I}.\label{eq:M^s_t}
\end{equation}
 By formula \eqref{eq:inverse}, we have 
\[
\left(\frac{d}{d\tau}\mathbf{N}_{\tau}^{t}\right)\mathbf{M}_{t}^{\tau}+\mathbf{N}_{\tau}^{t}\left(\frac{d}{d\tau}\mathbf{M}_{t}^{\tau}\right)=0,
\]
which together with \eqref{eq:M^s_t} yields 
\begin{equation}
\frac{d}{d\tau}\mathbf{N}_{\tau}^{t}=-\mathbf{N}_{\tau}^{t}\left[\mathbf{O}_{\tau}^{t}\mathbf{A}^{+}(\tau)\mathbf{O}_{t}^{\tau}\right],\qquad\mathbf{N}_{t}^{t}=\mathbf{I}.\label{eq:ODEs+}
\end{equation}
Taking the transpose of the expressions involved in the initial value
problem \eqref{eq:ODEs+} proves the last equation in \eqref{eq:ODEs-1}.
Finally, the uniqueness of both decompositions in \eqref{eq:DPD}
follows from the uniqueness of the solutions of the initial value
problems in \eqref{eq:ODEs-1}.

\section{Appendix C: Proof of Theorem \ref{theo:FDPD}}

Statements (i)-(iii) follow by a direct application of Theorem \ref{theo:DPD}
to the process $\mathbf{T}_{\tau}^{t}=\mathbf{\mathbf{F}}_{\tau}^{t}$.
To prove statement (iv), we apply the time-dependent coordinate change
\eqref{eq:objective} to the expression $\mathbf{\mathbf{x}}(t)=\mathbf{X}_{\tau}^{t}(\mathbf{\mathbf{x}}_{\tau})$
and obtain
\begin{equation}
\mathbf{y}(t)=\mathbf{Q}^{T}(t)\left[\mathbf{\mathbf{X}}_{\tau}^{t}(\mathbf{Q}(\tau)\mathbf{y}_{\tau}+\mathbf{b}(\tau))-\mathbf{b}(t)\right].\label{eq:ytransf}
\end{equation}
Differentiation of this equation with respect to $\mathbf{y}_{\tau}$
yields the transformed deformation gradient $\tilde{\mathbf{\mathbf{F}}}_{\tau}^{t}=\partial_{\mathbf{y}_{\tau}}\mathbf{y}(t)$
in the form 
\begin{equation}
\tilde{\mathbf{\mathbf{F}}}_{\tau}^{t}=\mathbf{Q}^{T}(t)\mathbf{\mathbf{F}}_{\tau}^{t}\mathbf{Q}(\tau),\label{eq:transformeddefgrad}
\end{equation}
showing that the deformation gradient tensor is not objective (cf.
Liu \cite{liu04}). Differentiating \eqref{eq:transformeddefgrad}
with respect to time, and first subtracting then adding the transpose
of the resulting equation, yields the transformed spin and rate-of-strain
tensors
\begin{equation}
\tilde{\mathbf{W}}(\mathbf{y},t)=\mathbf{Q}^{T}(t)\mathbf{W}(\mathbf{x},t)\mathbf{Q}(t)-\mathbf{Q}^{T}(t)\dot{\mathbf{Q}}(t),\qquad\tilde{\mathbf{D}}\left(\mathbf{y},t\right)=\mathbf{Q}^{T}(t)\mathbf{D}(\mathbf{x},t)\mathbf{Q}(t),\label{eq:transfSOmega}
\end{equation}
respectively, indicating that $\mathbf{\mathbf{W}}$ is not objective
but $\mathbf{D}$ is objective. 

Using the decomposition $\mathbf{\mathbf{F}}_{\tau}^{t}$ obtained
from statement (i) in the original $\mathbf{\mathbf{x}}$-frame, we
factorize the transformed deformation gradient \eqref{eq:transformeddefgrad}
as 
\begin{equation}
\tilde{\mathbf{\mathbf{F}}}_{\tau}^{t}=\mathbf{Q}^{T}(t)\mathbf{O}_{\tau}^{t}\mathbf{M}_{\tau}^{t}\mathbf{Q}(\tau)=\tilde{\mathbf{O}}_{\tau}^{t}\tilde{\mathbf{M}}_{\tau}^{t},\qquad\tilde{\mathbf{O}}_{\tau}^{t}=\mathbf{Q}^{T}(t)\mathbf{O}_{\tau}^{t}\mathbf{Q}(\tau),\qquad\tilde{\mathbf{M}}_{\tau}^{t}=\mathbf{Q}^{T}(\tau)\mathbf{M}_{\tau}^{t}\mathbf{Q}(\tau).\label{eq:transfDPD-1}
\end{equation}
We want to show that this factorization is in fact the unique DPD
of the transformed deformation gradient $\tilde{\mathbf{\mathbf{F}}}_{\tau}^{t}$. 

To this end, note that
\begin{eqnarray}
\dot{\tilde{\mathbf{O}}}_{\tau}^{t} & = & \mathbf{Q}^{T}(t)\mathbf{\dot{O}}_{\tau}^{t}\mathbf{Q}(\tau)+\dot{\mathbf{Q}}^{T}(t)\mathbf{O}_{\tau}^{t}\mathbf{Q}(\tau)\nonumber \\
 & = & \mathbf{Q}^{T}(t)\mathbf{W}(\mathbf{x}(t),t)\mathbf{O}_{\tau}^{t}\mathbf{Q}(\tau)+\dot{\mathbf{Q}}^{T}(t)\mathbf{O}_{\tau}^{t}\mathbf{Q}(\tau)\nonumber \\
 & = & \tilde{\mathbf{W}}(\mathbf{y}(t),t)\tilde{\mathbf{O}}_{\tau}^{t},\nonumber \\
\dot{\tilde{\mathbf{M}}}_{\tau}^{t} & = & \mathbf{Q}^{T}(\tau)\mathbf{\dot{M}}_{\tau}^{t}\mathbf{Q}(\tau)\nonumber \\
 & = & \mathbf{Q}^{T}(\tau)\mathbf{O}_{t}^{\tau}\mathbf{D}(\mathbf{x}(t),t)\mathbf{O}_{\tau}^{t}\mathbf{M}_{\tau}^{t}\mathbf{Q}(\tau)\nonumber \\
 & = & \left[\tilde{\mathbf{O}}_{t}^{\tau}\tilde{\mathbf{D}}\left(\mathbf{y}(t),t\right)\tilde{\mathbf{O}}_{\tau}^{t}\right]\tilde{\mathbf{M}}_{\tau}^{t},\label{eq:transfDPD}
\end{eqnarray}
where we have used the identity $\dot{\mathbf{\mathbf{Q}}}^{T}\mathbf{\mathbf{Q}}=-\mathbf{\mathbf{Q}}^{T}\dot{\mathbf{\mathbf{Q}}}$
and the formulas from \eqref{eq:transfSOmega}. Therefore, by \eqref{eq:transfDPD},
$\tilde{\mathbf{\mathbf{O}}}_{\tau}^{t}$ is a rotational process
and $\tilde{\mathbf{\mathbf{M}}}_{\tau}^{t}$ is an irrotational family
of operators. By the uniqueness of the DPD, we conclude that \eqref{eq:transfDPD-1}
indeed represent the unique dynamic polar decomposition of the transformed
deformation gradient $\tilde{\mathbf{\mathbf{F}}}_{\tau}^{t}$.  By
the relation $\tilde{\mathbf{\mathbf{M}}}_{\tau}^{t}=\mathbf{\mathbf{Q}}^{T}(\tau)\mathbf{\mathbf{M}}_{\tau}^{t}\mathbf{\mathbf{Q}}(\tau)$,
the transformed dynamic stretch tensor $\tilde{\mathbf{\mathbf{M}}}_{\tau}^{t}$
is related to its original counterpart $\mathbf{\mathbf{M}}_{\tau}^{t}$
through a similarity transformation, and hence all scalar invariants
of $\mathbf{\mathbf{M}}_{\tau}^{t}$ are preserved in the new frame.
We do not, however, have $\tilde{\mathbf{\mathbf{M}}}_{\tau}^{t}=\mathbf{\mathbf{Q}}^{T}(t)\mathbf{\mathbf{M}}_{\tau}^{t}\mathbf{\mathbf{Q}}(t)$
and hence $\mathbf{\mathbf{M}}_{\tau}^{t}$ is not objective. Finally,
rewriting the transformed deformation gradient as

\begin{equation}
\tilde{\mathbf{\mathbf{F}}}_{\tau}^{t}=\mathbf{Q}^{T}(t)\mathbf{N}_{\tau}^{t}\mathbf{O}_{\tau}^{t}\mathbf{Q}(\tau)=\tilde{\mathbf{N}}_{\tau}^{t}\tilde{\mathbf{O}}_{\tau}^{t},\qquad\tilde{\mathbf{N}}_{\tau}^{t}=\mathbf{Q}^{T}(t)\mathbf{N}_{\tau}^{t}\mathbf{Q}(t),\qquad\tilde{\mathbf{O}}_{\tau}^{t}=\mathbf{Q}^{T}(t)\mathbf{O}_{\tau}^{t}\mathbf{Q}(\tau),\label{eq:transfDPD-2}
\end{equation}
and repeating the rest of the above argument for the left dynamic
stretch tensor $\mathbf{\mathbf{N}}_{\tau}^{t}$ completes the proof
of statement (iv). Note that $\mathbf{N}_{\tau}^{t}$ is objective
by the second formula in \eqref{eq:transfDPD-2}.

\section{Appendix D: Proof of Theorem \ref{theo:relativerot}}

To prove the first decomposition in \eqref{eq:object_decomp}, we
write the rotation tensor $\mathbf{\mathbf{O}}_{\tau}^{t}$ in the
form $\mathbf{\mathbf{O}}_{\tau}^{t}=\mathbf{\boldsymbol{\Phi}}_{\tau}^{t}\mathbf{\boldsymbol{\Theta}}_{\tau}^{t}$,
with $\mathbf{\mathbf{\boldsymbol{\Phi}}}_{\tau}^{t}$,$\mathbf{\boldsymbol{\Theta}}_{\tau}^{t}$$\in SO(3)$,
$\mathbf{\mathbf{\boldsymbol{\Phi}}}_{\tau}^{\tau}=\mathbf{\boldsymbol{\Theta}}_{\tau}^{\tau}=\mathbf{I}$
yet to be determined. Differentiating this factorization with respect
to $t$ gives
\begin{equation}
\dot{\mathbf{O}}_{\tau}^{t}=\dot{\mathbf{\mathbf{\boldsymbol{\Phi}}}}_{\tau}^{t}\mathbf{\boldsymbol{\Theta}}_{\tau}^{t}+\mathbf{\mathbf{\boldsymbol{\Phi}}}_{\tau}^{t}\dot{\mathbf{\boldsymbol{\Theta}}}_{\tau}^{t}.\label{eq:objdecomp1}
\end{equation}
 At the same time, we also rewrite the ODE \eqref{eq:OODE} defining
$\mathbf{\mathbf{O}}_{\tau}^{t}$ in the form
\begin{equation}
\dot{\mathbf{O}}_{\tau}^{t}=\left[\mathbf{W}(\mathbf{x}(t),t)-\bar{\mathbf{W}}(t)\right]\mathbf{\mathbf{\boldsymbol{\Phi}}}_{\tau}^{t}\mathbf{\boldsymbol{\Theta}}_{\tau}^{t}+\bar{\mathbf{W}}(t)\mathbf{\mathbf{\boldsymbol{\Phi}}}_{\tau}^{t}\mathbf{\boldsymbol{\Theta}}_{\tau}^{t}.\label{eq:objdecomp2}
\end{equation}
Equating the first and second terms in the right-hand sides of \eqref{eq:objdecomp1}
and \eqref{eq:objdecomp2} leads to the initial value problems \eqref{eq:PHIODE}-\eqref{eq:THETAODE}
proving the uniqueness of the first decomposition in statement (i).
The relative rotation tensor is a rotational process, given that it
is the fundamental matrix solution of the classical system of ODEs
\eqref{eq:PHIODE}, whose skew-symmetric right-hand side has no explicit
dependence on the initial time $\tau$. By \eqref{eq:THETAODE}, the
mean-rotation tensor $\mathbf{\boldsymbol{\Theta}}_{\tau}^{t}$ forms
a rotational operator family. However, $\mathbf{\boldsymbol{\Theta}}_{\tau}^{t}$
is generally not a linear process, given the explicit dependence of
the right-hand side of \eqref{eq:THETAODE} on the initial time $\tau$.

To prove the second decomposition of $\mathbf{\mathbf{O}}_{\tau}^{t}$
in \eqref{eq:object_decomp}, we observe that 
\[
\mathbf{O}_{\tau}^{t}=\left[\mathbf{O}_{t}^{\tau}\right]^{T}=\left[\mathbf{\mathbf{\boldsymbol{\Phi}}}_{t}^{\tau}\mathbf{\boldsymbol{\Theta}}_{t}^{\tau}\right]^{T}=\left(\mathbf{\boldsymbol{\Theta}}_{t}^{\tau}\right)^{T}\left(\mathbf{\mathbf{\boldsymbol{\Phi}}}_{t}^{\tau}\right)^{T}=\left(\mathbf{\boldsymbol{\Theta}}_{t}^{\tau}\right)^{T}\mathbf{\mathbf{\boldsymbol{\Phi}}}_{\tau}^{t},
\]
thus setting 
\begin{equation}
\mathbf{\boldsymbol{\Sigma}}_{\tau}^{t}=\left(\mathbf{\boldsymbol{\Theta}}_{t}^{\tau}\right)^{T}\in SO(3),\label{eq:inverse-1}
\end{equation}
we recover the left mean rotation tensor $\mathbf{\boldsymbol{\Sigma}}_{\tau}^{t}$,
as claimed.  Interchanging the role of $\tau$ and $t$ in the second
equation of \eqref{eq:ODEs-1}, we find that 
\begin{equation}
\frac{d}{d\tau}\mathbf{\boldsymbol{\Theta}}_{t}^{\tau}=\left[\mathbf{\mathbf{\boldsymbol{\Phi}}}_{\tau}^{t}\bar{\mathbf{W}}(\tau)\mathbf{\mathbf{\boldsymbol{\Phi}}}_{t}^{\tau}\right]\mathbf{\boldsymbol{\Theta}}_{t}^{\tau},\qquad\mathbf{\boldsymbol{\Theta}}_{t}^{t}=\mathbf{I},\label{eq:M^s_t-1}
\end{equation}
thus, using formula \eqref{eq:inverse-1}, we obtain 
\[
\left(\frac{d}{d\tau}\mathbf{\boldsymbol{\Sigma}}_{\tau}^{t}\right)\mathbf{\boldsymbol{\Theta}}_{t}^{\tau}+\mathbf{\boldsymbol{\Sigma}}_{\tau}^{t}\left(\frac{d}{d\tau}\mathbf{\boldsymbol{\Theta}}_{t}^{\tau}\right)=0.
\]
This last equation together with \eqref{eq:M^s_t-1} gives the initial
value problem 
\begin{equation}
\frac{d}{d\tau}\mathbf{\boldsymbol{\Sigma}}_{\tau}^{t}=-\mathbf{\boldsymbol{\Sigma}}_{\tau}^{t}\left[\mathbf{\mathbf{\boldsymbol{\Phi}}}_{\tau}^{t}\bar{\mathbf{W}}(\tau)\mathbf{\mathbf{\boldsymbol{\Phi}}}_{t}^{\tau}\right],\qquad\mathbf{\boldsymbol{\Sigma}}_{t}^{t}=\mathbf{I}.\label{eq:ODEs+-1}
\end{equation}
Taking the transpose of \eqref{eq:ODEs+-1} proves the last equation
in \eqref{eq:SIGMAODE}. Again, the uniqueness of both decompositions
in \eqref{eq:object_decomp} follows from the uniqueness of solutions
to \eqref{eq:SIGMAODE}. Finally, $\mathbf{\boldsymbol{\Sigma}}_{\tau}^{t}$
is a rotational operator family, but not a process, as discussed already
for $\mathbf{\boldsymbol{\Theta}}_{t}^{\tau}$.

To prove the last statement of the theorem, we first change coordinates
under a general Euclidean transformation \eqref{eq:objective}, and
use tilde, as in the proof of Theorem \ref{theo:FDPD}, to denote
quantities in the $\mathbf{y}$ coordinate frame. We recall from formula
\eqref{eq:transfSOmega} the form of the transformed vorticity tensor
\begin{equation}
\tilde{\mathbf{W}}(\mathbf{y},t)=\mathbf{Q}^{T}(t)\mathbf{W}(\mathbf{x},t)\mathbf{Q}(t)-\mathbf{Q}^{T}(t)\dot{\mathbf{Q}}(t).\label{eq:transfSOmega-1}
\end{equation}
Taking the spatial mean of both sides in eq. \eqref{eq:transfSOmega-1}
over the body $\mathcal{B}(t)$, and noting that the transformation
\eqref{eq:objective} preserves the volume of $\mathcal{B}(t)$, we
obtain 
\begin{equation}
\bar{\tilde{\mathbf{W}}}(t)=\mathbf{Q}^{T}(t)\bar{\mathbf{W}}(t)\mathbf{Q}(t)-\mathbf{Q}^{T}(t)\dot{\mathbf{Q}}(t).\label{eq:transfSOmega-2}
\end{equation}
Subtracting \eqref{eq:transfSOmega-2} from \eqref{eq:transfSOmega-1}
gives
\begin{equation}
\tilde{\mathbf{W}}(\mathbf{y},t)-\bar{\tilde{\mathbf{W}}}(t)=\mathbf{Q}^{T}(t)\left[\mathbf{W}(\mathbf{x},t)-\mathbf{\bar{W}}(t)\right]\mathbf{Q}(t).\label{eq:transfSOmega-3}
\end{equation}

Next, using the decomposition of $\mathbf{\mathbf{O}}_{\tau}^{t}$
from statement (i) in the original $\mathbf{\mathbf{x}}$-frame, we
factorize the transformed dynamic rotation tensor $\tilde{\mathbf{\mathbf{O}}}_{\tau}^{t}$
obtained in eq. \eqref{eq:transfDPD-1} as 
\begin{equation}
\tilde{\mathbf{O}}_{\tau}^{t}=\tilde{\mathbf{\mathbf{\boldsymbol{\Phi}}}}_{\tau}^{t}\tilde{\mathbf{\boldsymbol{\Theta}}}_{\tau}^{t},\qquad\tilde{\mathbf{\mathbf{\boldsymbol{\Phi}}}}_{\tau}^{t}=\mathbf{Q}^{T}(t)\mathbf{\mathbf{\boldsymbol{\Phi}}}_{\tau}^{t}\mathbf{P}(t),\qquad\tilde{\mathbf{\boldsymbol{\Theta}}}_{\tau}^{t}=\mathbf{P}^{-1}(t)\mathbf{\boldsymbol{\Theta}}_{\tau}^{t}\mathbf{Q}(\tau),\label{eq:transfDPD-1-1}
\end{equation}
with the matrix $\mathbf{P}(t)$ to be determined in a way that \eqref{eq:transfDPD-1-1}
gives the unique relative-mean rotation decomposition of $\tilde{\mathbf{\mathbf{O}}}_{\tau}^{t}$
in the $\mathbf{y}$ coordinate frame. Both $\tilde{\mathbf{\mathbf{\boldsymbol{\Phi}}}}_{\tau}^{t}$
and $\mathbf{\mathbf{\boldsymbol{\Phi}}}_{\tau}^{t}$, as well as
$\tilde{\mathbf{\boldsymbol{\Theta}}}_{\tau}^{t}$ and $\mathbf{\boldsymbol{\Theta}}_{\tau}^{t}$,
are equal to the identity matrix at time $t=\tau,$ thus by \eqref{eq:transfDPD-1-1},
$\mathbf{P}(t)$ must necessarily satisfy
\begin{equation}
\mathbf{P}(\tau)=\mathbf{Q}(\tau).\label{eq:At0}
\end{equation}

To determine $\mathbf{P}(t)$, we differentiate the expression for
$\tilde{\mathbf{\mathbf{\boldsymbol{\Phi}}}}_{\tau}^{t}$ in \eqref{eq:transfDPD-1-1},
then use \eqref{eq:PHIODE} and \eqref{eq:transfSOmega-3} to obtain
\begin{eqnarray}
\dot{\tilde{\mathbf{\boldsymbol{\Phi}}}}_{\tau}^{t} & = & \dot{\mathbf{Q}}^{T}(t)\mathbf{\mathbf{\boldsymbol{\Phi}}}_{\tau}^{t}\mathbf{P}(t)+\mathbf{Q}^{T}(t)\mathbf{\dot{\mathbf{\boldsymbol{\Phi}}}}_{\tau}^{t}\mathbf{P}(t)+\mathbf{Q}^{T}(t)\mathbf{\mathbf{\boldsymbol{\Phi}}}_{\tau}^{t}\dot{\mathbf{P}}(t),\nonumber \\
 & = & \dot{\mathbf{Q}}^{T}(t)\mathbf{\mathbf{\boldsymbol{\Phi}}}_{\tau}^{t}\mathbf{P}(t)+\mathbf{Q}^{T}(t)\left[\mathbf{W}(\mathbf{x},t)-\bar{\mathbf{W}}(t)\right]\mathbf{\mathbf{\boldsymbol{\Phi}}}_{\tau}^{t}\mathbf{P}(t)+\mathbf{Q}^{T}(t)\mathbf{\mathbf{\boldsymbol{\Phi}}}_{\tau}^{t}\dot{\mathbf{P}}(t)\nonumber \\
 & = & \left[\tilde{\mathbf{W}}(\mathbf{y},t)-\bar{\tilde{\mathbf{W}}}(t)\right]\tilde{\mathbf{\mathbf{\boldsymbol{\Phi}}}}_{\tau}^{t}+\dot{\mathbf{Q}}^{T}(t)\mathbf{\mathbf{\boldsymbol{\Phi}}}_{\tau}^{t}\mathbf{P}(t)+\mathbf{Q}^{T}(t)\mathbf{\mathbf{\boldsymbol{\Phi}}}_{\tau}^{t}\dot{\mathbf{P}}(t).\label{eq:finobj1}
\end{eqnarray}
The transformed relative rotation tensor $\tilde{\mathbf{\mathbf{\boldsymbol{\Phi}}}}_{\tau}^{t}$
is defined by the equation $\dot{\tilde{\mathbf{\boldsymbol{\Phi}}}}_{\tau}^{t}=\left[\tilde{\mathbf{W}}(\mathbf{y},t)-\bar{\tilde{\mathbf{W}}}(t)\right]\tilde{\mathbf{\mathbf{\boldsymbol{\Phi}}}}_{\tau}^{t}$
in the $\mathbf{y}$ coordinates, therefore \eqref{eq:finobj1} implies
\[
\dot{\mathbf{Q}}^{T}(t)\mathbf{\mathbf{\boldsymbol{\Phi}}}_{\tau}^{t}\mathbf{P}(t)+\mathbf{Q}^{T}(t)\mathbf{\mathbf{\boldsymbol{\Phi}}}_{\tau}^{t}\dot{\mathbf{P}}(t)=\mathbf{0},
\]
or, equivalently, 
\begin{equation}
\dot{\mathbf{P}}(t)=\mathbf{\mathbf{\boldsymbol{\Phi}}}_{t}^{\tau}\dot{\mathbf{Q}}(t)\mathbf{Q}^{T}(t)\mathbf{\mathbf{\boldsymbol{\Phi}}}_{\tau}^{t}\mathbf{P}(t).\label{eq:Aform1}
\end{equation}
 This linear system of differential equations has a skew-symmetric
coefficient matrix, therefore $\mathbf{P}(t)$ is a proper orthogonal
matrix, and hence
\begin{equation}
\mathbf{P}^{-1}(t)=\mathbf{P}^{T}(t).\label{eq:Aform2}
\end{equation}

For two-dimensional deformations, the skew-symmetric tensor $\dot{\mathbf{\mathbf{Q}}}^{T}(t)\mathbf{\mathbf{Q}}(t)$
is always a scalar multiple of a rotation tensor, and hence commutes
with any other two-dimensional rotation tensor. Consequently, equation
\eqref{eq:Aform1} can be re-written as 
\begin{equation}
\dot{\mathbf{Q}}^{T}(t)\mathbf{\mathbf{\boldsymbol{\Phi}}}_{\tau}^{t}\mathbf{P}(t)+\mathbf{Q}^{T}(t)\mathbf{\mathbf{\boldsymbol{\Phi}}}_{\tau}^{t}\dot{\mathbf{P}}(t)=\mathbf{\mathbf{\boldsymbol{\Phi}}}_{\tau}^{t}\left[\dot{\mathbf{Q}}^{T}(t)\mathbf{P}(t)+\mathbf{Q}^{T}(t)\dot{\mathbf{P}}(t)\right]=\mathbf{\mathbf{\boldsymbol{\Phi}}}_{\tau}^{t}\frac{d}{dt}\left[\mathbf{Q}^{T}(t)\mathbf{P}(t)\right]=\mathbf{0},\label{eq:finobj2}
\end{equation}
implying that $\mathbf{Q}^{T}(t)\mathbf{P}(t)$ is a constant rotation.
Therefore, by \eqref{eq:At0}, we conclude from \eqref{eq:finobj2}
for two-dimensional deformations that $\mathbf{P}(t)\equiv\mathbf{Q}(t).$
Thus formula \eqref{eq:transfDPD-1-1} gives 
\[
\tilde{\mathbf{\mathbf{\boldsymbol{\Phi}}}}_{\tau}^{t}=\mathbf{Q}^{T}(t)\mathbf{\mathbf{\boldsymbol{\Phi}}}_{\tau}^{t}\mathbf{Q}(t),
\]
proving statement (ii) of Theorem \ref{theo:relativerot}.

\section{Appendix E: Fiber-averaged angular velocity of a rigid body}

Consider a perfectly rigid body $\mathcal{R}(t)$, with a well-defined
angular velocity vector $\mathbf{\boldsymbol{\nu}}_{rigid}(t)$ (see
Fig. \ref{fig:nu_min}a). 
\begin{figure}[H]
\begin{centering}
\includegraphics[width=0.7\textwidth]{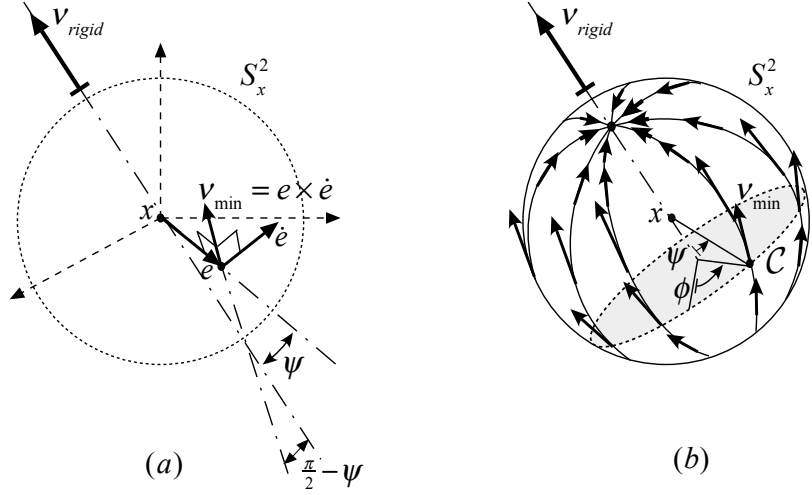}
\par\end{centering}

\caption{(a) The geometry of the minimal admissible angular velocity $\mathbf{\boldsymbol{\nu}}_{\min}\left(\mathbf{\mathbf{x}},t;\mathbf{e}\right)$
at a point $\mathbf{\mathbf{x}}$ and the actual angular velocity
$\mathbf{\boldsymbol{\nu}}_{rigid}(t)$ in case of an ideal rigid
body motion. (b) The radial components of the vector field $\mathbf{\boldsymbol{\nu}}_{\min}\left(\mathbf{\mathbf{x}},t;\mathbf{e}\right)$
along the circle $\mathcal{C}$ average out to zero, and hence only
the components normal to the plane of $\mathcal{C}$ contribute to
the average $\left\langle \mathbf{\boldsymbol{\nu}}_{\min}(\mathbf{\mathbf{x}},t,\mathbf{e})\right\rangle _{\mathbf{e}\in S_{\mathbf{x}}^{2}}$. }
\label{fig:nu_min}
\end{figure}
We seek to average $\mathbf{\boldsymbol{\nu}}_{\min}\left(\mathbf{\mathbf{x}},t;\mathbf{e}\right)$
over all vectors $\mathbf{e}(\phi,\psi)$ taken from the spherically
parametrized unit sphere $S_{\mathbf{\mathbf{x}}}^{2}$. Note the
cancellation of the averaged vector in radial directions normal to
$\mathbf{\boldsymbol{\nu}}_{rigid}(t)$ due to the circular symmetry
shown in Fig. \ref{fig:nu_min}b. Further note from Fig. \ref{fig:nu_min}a
that the projection of $\mathbf{\boldsymbol{\nu}}_{\min}\left(\mathbf{\mathbf{x}},t;\mathbf{e}\right)$
on the axis of rotation defined by $\mathbf{\boldsymbol{\nu}}_{rigid}(t)$
is 
\[
\left|\mathbf{\boldsymbol{\nu}}_{\min}\left(\mathbf{x},t;\mathbf{e}\right)\right|\sin\psi=\left|\mathbf{\boldsymbol{\nu}}_{rigid}(t)\right|\sin^{2}\psi.
\]
From these considerations, we obtain that the average of the vector
field $\mathbf{\boldsymbol{\nu}}_{\min}\left(\mathbf{\mathbf{x}},t;\mathbf{e}\right)$
over $S_{\mathbf{\mathbf{x}}}^{2}$ is 
\begin{eqnarray}
\left\langle \mathbf{\boldsymbol{\nu}}_{\min}(\mathbf{x},t,\mathbf{e})\right\rangle _{\mathbf{e}\in S_{\mathbf{x}}^{2}} & = & \frac{1}{2\pi}\intop_{0}^{2\pi}\left[\frac{1}{\pi}\intop_{0}^{\pi}\mathbf{\boldsymbol{\nu}}_{\min}\left(\mathbf{x},t;\mathbf{e}(\phi,\psi)\right)\,d\psi\right]d\phi=\frac{1}{\pi}\intop_{0}^{\pi}\mathbf{\boldsymbol{\nu}}_{rigid}(t)\sin^{2}\psi\,d\psi\label{eq:nurigid-1}\\
 & = & \frac{1}{2}\mathbf{\boldsymbol{\nu}}_{rigid}(t).
\end{eqnarray}
Therefore, for the material fiber-averaged angular velocity $\mathbf{\boldsymbol{\nu}}(t,\mathbf{\mathbf{x}})$
defined in \eqref{eq:nugeneral}, we obtain
\begin{equation}
\mathbf{\boldsymbol{\nu}}(t,\mathbf{x})=\mathbf{\boldsymbol{\nu}}_{rigid}(t)\label{eq:nu=00003Dnu_rigid}
\end{equation}
 in the case of a perfectly rigid body.

\section{Appendix F: Proof of Proposition \ref{prop:Fiber averaged 3D}}

In order to calculate the fiber-averaged angular velocity $\mathbf{\boldsymbol{\nu}}(\mathbf{\mathbf{x}},t)$
defined in \eqref{eq:nugeneral}, we first need a general expression
for the derivative $\mathbf{\dot{e}}(t)$ for an arbitrary unit vector
$\mathbf{e}(t)$ tangent to an evolving material fiber. Differentiating
the definition \eqref{eq:def_e(t)} of $\mathbf{e}(t)$ in time, and
using Example \ref{ex:The-deformation-gradient}, we obtain 
\[
\dot{\mathbf{e}}=\frac{\mathbf{\nabla v}(\mathbf{x}(t),t)\mathbf{\mathbf{F}}_{\tau}^{t}\mathbf{e}(\tau)-\mathbf{\mathbf{F}}_{\tau}^{t}\mathbf{e}(\tau)\frac{\left\langle \mathbf{\nabla v}(\mathbf{x}(t),t)\mathbf{\mathbf{F}}_{\tau}^{t}\mathbf{e}(\tau)\,,\,\mathbf{\mathbf{F}}_{\tau}^{t}\mathbf{e}(\tau)\right\rangle }{\left|\mathbf{\mathbf{F}}_{\tau}^{t}\mathbf{e}(\tau)\right|}}{\left|\mathbf{\mathbf{F}}_{\tau}^{t}\mathbf{e}(\tau)\right|^{2}}.
\]
 Setting $\tau$ equal to $t$ in this last equation and using formula
\eqref{eq:WDdef} gives
\begin{equation}
\dot{\mathbf{e}}=\left[\mathbf{W}+\mathbf{D}-\left\langle \mathbf{e},\mathbf{D}\mathbf{e}\right\rangle \mathbf{I}\right]\mathbf{e}.\label{eq:edotform}
\end{equation}
This equation is broadly known (see, e.g., Chadwick \cite{Chadwick76}),
and has only been re-derived here for completeness and notational
consistence. 

Taking the cross product of both sides with $\mathbf{e}$ and using
the definitions \eqref{eq:nu_min} and \eqref{eq:nugeneral}, we obtain
from \eqref{eq:edotform} the general expression 
\begin{eqnarray}
\mathbf{\boldsymbol{\nu}} & = & 2\left\langle \mathbf{e}\times\mathbf{W}\mathbf{e}\right\rangle _{\mathbf{e}\in S_{\mathbf{x}}^{2}}+2\left\langle \mathbf{e}\times\mathbf{De}\right\rangle _{\mathbf{e}\in S_{\mathbf{x}}^{2}}\nonumber \\
 & = & \frac{1}{2}\mathbf{\boldsymbol{\omega}}+2\left\langle \mathbf{e}\times\mathbf{D}\mathbf{e}\right\rangle _{\mathbf{e}\in S_{\mathbf{x}}^{2}},\label{eq:nuxt_average}
\end{eqnarray}
where we have applied the relationship \eqref{eq:nu=00003Dnu_rigid}
to the rigid body rotation generated by the angular velocity tensor
$\mathbf{\mathbf{W}}(\mathbf{\mathbf{x}},t)$ with angular velocity
$\mathbf{\boldsymbol{\nu}}_{rigid}(t)=\frac{1}{2}\boldsymbol{\omega}(\mathbf{\mathbf{x}},t)$. 

Let $\{\mathbf{b}_{i}(\mathbf{\mathbf{x}},t)\}_{i=1}^{3}$ denote
a positively oriented orthonormal basis for the rate-of-strain tensor
$\mathbf{D}(\mathbf{\mathbf{x}},t)$, with corresponding eigenvalues
$\sigma_{1}(\mathbf{\mathbf{x}},t)\leq\sigma_{2}(\mathbf{\mathbf{x}},t)\leq\sigma_{3}(\mathbf{\mathbf{x}},t).$
In this basis, the unit vector $\mathbf{e}$ has the classic spherical
coordinate representation (cf. Fig. \ref{fig:nu_min}b) 
\[
\mathbf{e}=\cos\psi\cos\phi\mathbf{b}_{1}+\cos\psi\sin\phi\mathbf{b}_{2}+\sin\psi\mathbf{b}_{3},
\]
from which we obtain
\[
\mathbf{e}\times\mathbf{De}=\frac{1}{2}\left(\sigma_{3}-\sigma_{2}\right)\sin2\psi\sin\phi\mathbf{b}_{1}+\frac{1}{2}\left(\sigma_{2}-\sigma_{3}\right)\sin2\psi\cos\phi\mathbf{b}_{2}+\frac{1}{2}\left(\sigma_{2}-\sigma_{1}\right)\sin2\phi\cos^{2}\phi\mathbf{b}_{3}.
\]
This shows that 
\[
\left\langle \mathbf{e}\times\mathbf{De}\right\rangle _{\mathbf{e}\in S_{\mathbf{x}}^{2}}=0,
\]
thus formula \eqref{eq:nuxt_average} simplifies to $\boldsymbol{\nu}(\mathbf{\mathbf{x}},t)=\frac{1}{2}\mathbf{\boldsymbol{\omega}}(\mathbf{\mathbf{x}},t)$,
proving the statement of Proposition \ref{prop:Fiber averaged 3D}.

\section{Appendix G: Proof of Theorem \ref{theo:dynamic angles}}

By Theorem 2, we have 
\begin{eqnarray*}
\mathbf{\dot{O}}_{\tau}^{t}\left[\mathbf{O}_{\tau}^{t}\right]^{T}\mathbf{e} & = & \mathbf{W}(\mathbf{x}(t),t)\mathbf{e=}-\frac{1}{2}\mathbf{\omega}(\mathbf{x}(t),t)\mathbf{\times e},\\
\mathbf{\dot{\Phi}}_{\tau}^{t}\left[\mathbf{\Phi}_{\tau}^{t}\right]^{T}\mathbf{e} & = & \left[\mathbf{W}(\mathbf{x}(t),t)-\mathbf{\bar{W}}(t)\right]\mathbf{e=}-\frac{1}{2}\left[\mathbf{\omega}(\mathbf{x}(t),t)-\bar{\mathbf{\omega}}(t)\right]\times\mathbf{e}.
\end{eqnarray*}
Therefore, 
\begin{eqnarray*}
\varphi_{\tau}^{t}(\mathbf{x}_{\tau};\mathbf{g}) & = & -\frac{1}{2}\int_{\tau}^{t}\mathbf{\boldsymbol{\omega}}(\mathbf{x}(s),s)\cdot\mathbf{g}(\mathbf{x}(s),s)\,ds\\
 & = & -\frac{1}{2}\int_{\sigma}^{t}\mathbf{\boldsymbol{\omega}}(\mathbf{x}(s),s)\cdot\mathbf{g}(\mathbf{x}(s),s)\,ds-\frac{1}{2}\int_{\tau}^{\sigma}\mathbf{\boldsymbol{\omega}}(\mathbf{x}(s),s)\cdot\mathbf{g}(\mathbf{x}(s),s)\,ds\\
 & = & \varphi_{\sigma}^{t}(\mathbf{x}_{\sigma};\mathbf{g})+\varphi_{\tau}^{\sigma}(\mathbf{x}_{\tau};\mathbf{g}),
\end{eqnarray*}
and similarly, 
\begin{equation}
\phi_{\tau}^{t}(\mathbf{x}_{\tau};\mathbf{g})=\phi_{\sigma}^{t}(\mathbf{x}_{\sigma};\mathbf{g})+\phi_{\tau}^{\sigma}(\mathbf{x}_{\tau};\mathbf{g}),\label{eq:addangle}
\end{equation}
proving the dynamical consistency of $\varphi_{\tau}^{t}$ and $\phi_{\tau}^{t}$,
and completing the proof of statement (i) of the theorem.

To complete the proof of statement (ii), we must prove the objectivity
of the relative dynamic rotation $\varphi_{\tau}^{t}(\mathbf{\mathbf{x}}_{\tau};\mathbf{g})$
under a Euclidean frame change of the form \eqref{eq:objective}.
As is well known (see., e.g., Truesdell \& Rajagopal \cite{Truesdell09}),
the transformed vorticity $\tilde{\mathbf{\boldsymbol{\omega}}}(\mathbf{y},t)$
is related to the original vorticity $\boldsymbol{\omega}(\mathbf{\mathbf{x}},t)$
through the formula 
\begin{equation}
\boldsymbol{\omega}(\mathbf{x},t)=\mathbf{Q}(t)\tilde{\boldsymbol{\omega}}(\mathbf{y},t)+\dot{\mathbf{q}}(t),\label{eq:vort trans}
\end{equation}
where the vector $\dot{\mathbf{q}}$ is defined via the identity $\frac{1}{2}\dot{\mathbf{q}}\times\tilde{\mathbf{a}}=\dot{\mathbf{\mathbf{Q}}}\mathbf{Q}^{T}\tilde{\mathbf{a}}$
for all $\tilde{\mathbf{a}}\in\mathbb{\mathbb{R}}^{3}$, accounting
for the additional vorticity introduced by the frame change. Taking
the spatial means of both sides in \eqref{eq:vort trans} over the
evolving continuum $\mathcal{B}(t)$ gives
\begin{equation}
\bar{\mathbf{\boldsymbol{\omega}}}(t)=\mathbf{Q}(t)\bar{\tilde{\mathbf{\boldsymbol{\omega}}}}(t)+\dot{\mathbf{q}}(t),\label{eq:vort trans plus}
\end{equation}
because the volume of $\mathcal{B}(t)$ remains constant under the
the Euclidean frame change \eqref{eq:objective}. Subtracting \eqref{eq:vort trans plus}
from \eqref{eq:vort trans}, we obtain that

\begin{equation}
\mathbf{\boldsymbol{\omega}}(\mathbf{x},t)-\bar{\mathbf{\boldsymbol{\omega}}}(t)=\mathbf{Q}(t)\left[\tilde{\mathbf{\boldsymbol{\omega}}}(\mathbf{y},t)-\bar{\tilde{\mathbf{\boldsymbol{\omega}}}}(t)\right].\label{eq: vort diff transf}
\end{equation}
The vector field $\mathbf{g}(\mathbf{\mathbf{x}},t)$ is transformed
under the frame change as
\begin{equation}
\tilde{\mathbf{g}}(\mathbf{y},t)=\mathbf{Q}^{T}(t)\mathbf{g}(\mathbf{x},t).\label{eq:transformed e}
\end{equation}
We observe that in the rotating frame, $\tilde{\mathbf{g}}$ is necessarily
time-dependent, even if $\mathbf{g}$ was originally chosen as a time-independent
constant direction. Using the formulas \eqref{eq: vort diff transf}
and \eqref{eq:transformed e}, we obtain that
\begin{eqnarray*}
\phi_{\tau}^{t}(\mathbf{x}_{\tau};\mathbf{g}) & = & -\frac{1}{2}\intop_{\tau}^{t}\left[\boldsymbol{\omega}(\mathbf{x}(s),s)-\bar{\mathbf{\boldsymbol{\omega}}}(s)\right]\cdot\mathbf{g}(\mathbf{x}(s),s)\,\,ds\\
 & =-\frac{1}{2} & \intop_{\tau}^{t}\mathbf{Q}(s)\left[\tilde{\mathbf{\boldsymbol{\omega}}}(\mathbf{y}(s),s)-\bar{\tilde{\mathbf{\boldsymbol{\omega}}}}(s)\right]\cdot\mathbf{Q}(s)\mathbf{g}(\mathbf{x}(s),s)\,\,ds,\\
 & = & -\frac{1}{2}\intop_{\tau}^{t}\left[\tilde{\mathbf{\boldsymbol{\omega}}}(\mathbf{y}(s),s)-\bar{\tilde{\mathbf{\boldsymbol{\omega}}}}(s)\right]\cdot\tilde{\mathbf{g}}(\mathbf{x}(s),s)\,\,ds\\
 & = & \phi_{\tau}^{t}(\mathbf{y}_{\tau};\tilde{\mathbf{g}}),
\end{eqnarray*}
which completes the proof of statement (ii) of the theorem. Statement
(iii) then follows by setting $\mathbf{g}=-\left(\mathbf{\omega}-\bar{\mathbf{\omega}}\right)/\left|\mathbf{\omega}-\bar{\mathbf{\omega}}\right|.$ 
\begin{rem}
\label{remark: angle argument for polar rotation}The argument leading
to \eqref{eq:addangle} would \emph{not} work for the polar rotation
angle. Indeed, the angular velocity $\dot{\mathbf{\boldsymbol{q}}}_{polar}(t,\tau)$
of the polar rotations inherits explicit dependence on $\tau$ from
$\dot{\mathbf{\mathbf{R}}}_{\tau}^{t}\left(\mathbf{\mathbf{R}}_{\tau}^{t}\right)^{T}$.
As a consequence, for the polar rotation angle defined as 
\[
\gamma_{\tau}^{t}(\mathbf{x}_{\tau};\mathbf{g})=\int_{\tau}^{t}\dot{\mathbf{\boldsymbol{q}}}_{polar}(s,\tau)\cdot\mathbf{g}(\mathbf{x}(s),s)\,ds,
\]
 we obtain 
\begin{eqnarray*}
\gamma_{\tau}^{t}(\mathbf{x}_{\tau};\mathbf{g}) & = & \int_{\sigma}^{t}\dot{\mathbf{\boldsymbol{q}}}_{polar}(s,\tau)\cdot\mathbf{g}(\mathbf{x}(s),s)\,ds+\int_{\tau}^{\sigma}\dot{\mathbf{\boldsymbol{q}}}_{polar}(s,\tau)\cdot\mathbf{g}(\mathbf{x}(s),s)\,ds\\
 & = & \int_{\sigma}^{t}\dot{\mathbf{\boldsymbol{q}}}_{polar}(s,\tau)\cdot\mathbf{g}(\mathbf{x}(s),s)\,ds+\gamma_{\tau}^{\sigma}(\mathbf{x}_{\tau};\mathbf{g})\\
 & \ne & \gamma_{\sigma}^{t}(\mathbf{x}_{\sigma};\mathbf{g})+\gamma_{\tau}^{\sigma}(\mathbf{x}_{\tau};\mathbf{g}),
\end{eqnarray*}
 because we generally have $\dot{\mathbf{\boldsymbol{q}}}_{polar}(s,\tau)\neq\dot{\mathbf{\boldsymbol{q}}}_{polar}(s,\sigma)$
for $\tau\neq\sigma.$\end{rem}

\end{document}